\documentclass[12pt,a4paper]{amsart}
\usepackage{amssymb,amscd}
\usepackage[mathcal]{eucal}
\theoremstyle{plain}
\newtheorem{thm}{Theorem}[section]
\newtheorem{lem}[thm]{Lemma}
\newtheorem{pro}[thm]{Proposition}
\newtheorem{cor}[thm]{Corollary}

\newtheorem{thmABC}{Theorem}

\newtheorem{corABC}[thmABC]{Corollary}

\theoremstyle{remark}

\theoremstyle{definition} 
\newtheorem{dfn}[thm]{Definition}
\newtheorem{exm}[thm]{Example}

\numberwithin{equation}{section}

\newcommand{\N}{\mathbb{N}}
\newcommand{\R}{\mathbb{R}}
\newcommand{\Q}{\mathbb{Q}}
\newcommand{\Z}{\mathbb{Z}}
\newcommand{\C}{\mathbb{C}}

\newcommand{\Ad}{\textup{Ad}}
\newcommand{\Id}{\mathrm{Id}}
\newcommand{\id}{\mathrm{id}}

\newcommand{\ac}[1]{\overline{#1}^{\, \textup{z}}} % algebraic closure

\renewcommand{\epsilon}{\varepsilon} 
\renewcommand{\phi}{\varphi} 
\renewcommand{\rho}{\varrho} 
\renewcommand{\theta}{\vartheta} 

\DeclareMathOperator{\Aut}{Aut}
\DeclareMathOperator{\Hom}{Hom}
\DeclareMathOperator{\Fitt}{Fitt}
\DeclareMathOperator{\Nil}{Nil}
\DeclareMathOperator{\rank}{rk} 
\DeclareMathOperator{\up}{u}
\DeclareMathOperator{\cdim}{cdim}
\DeclareMathOperator{\GL}{GL}
\DeclareMathOperator{\Rad}{Rad}
\DeclareMathOperator{\Cen}{C}
\DeclareMathOperator{\Zen}{Z}
\DeclareMathOperator{\Der}{Der}

\DeclareMathOperator{\diag}{diag}

\begin{document}
\title[Lattices in solvable Lie groups]{Deformations and rigidity of
  lattices \\ in solvable Lie groups}

\thanks{\textit{Acknowledgements.} The authors thank the Deutsche
  Forschungsgemeinschaft and the London Mathematical Society for
  financial support.}

\author{Oliver Baues} \address{Institut f\"ur Algebra und Geometrie,
  Karlsruher Institut f\"ur Technologie (KIT), 76128 Karlsruhe,
  Germany}
\email{baues@math.uni-karlsruhe.de}

\author{Benjamin Klopsch} \address{Department of Mathematics, Royal
  Holloway, University of London, Egham TW20 0EX, United Kingdom}
\email{Benjamin.Klopsch@rhul.ac.uk} 

\date{\today}

\begin{abstract}
%  The degree of
%  non-rigidity of a lattice $\Gamma$ in $G$ is captured by the
%  deformation space $\mathcal{D}(\Gamma,G)$. 
  Let $G$ be a simply connected, solvable Lie group and $\Gamma$ a lattice in $G$.  The deformation space $\mathcal{D}(\Gamma,G)$ is
  the orbit space % $\Aut(G) \backslash \mathcal{X}(\Gamma,G)$ 
  associated to the action of $\Aut(G)$ on the space $\mathcal{X}(\Gamma,G)$ of all lattice embeddings of $\Gamma$ into~$G$. 
  Our main result generalises the classical rigidity theorems of
  Mal'tsev and Sait\^o for lattices in nilpotent Lie groups and in
  solvable Lie groups of real type.  
  % In the given context it is natural to focus on Zariski-dense lattices.  
   We prove that the deformation space 
  of every Zariski-dense lattice $\Gamma$ in $G$ is finite and Hausdorff, provided that the maximal nilpotent normal subgroup of $G$ is connected. % Thereby we also show 
  This implies that every lattice in a solvable Lie group virtually embeds as a Zariski-dense lattice 
 with finite deformation space.  We give examples of solvable Lie groups~$G$ which admit Zariski-dense lattices~$\Gamma$ such that $\mathcal{D}(\Gamma,G)$ is countably infinite, and also examples where the maximal nilpotent normal subgroup of $G$ is connected  and 
 simultaneously $G$ has lattices with uncountable deformation space.
\end{abstract}



\maketitle

\section{Introduction}

Let $G$ be a simply connected real Lie group.  A lattice $\Gamma$ in
$G$ is \emph{rigid} if every embedding of $\Gamma$ into $G$ as a
lattice extends to an automorphism of the Lie group~$G$.  Landmark
results about rigidity and superrigidity of lattices in the context of
semisimple Lie groups are the Mostow Strong Rigidity Theorem and the
Margulis Superrigidity Theorem.

In the context of solvable Lie groups, a classical theorem of
Mal'tsev--Sait\^o states that, if $G$ is nilpotent \cite{Ma51} or
solvable of \emph{real type} \cite{Sa57}, i.e., if $G$ is solvable and
if the eigenvalues of all the transformations $\Ad_g$, $g \in G$, in
the adjoint action of $G$ are real, then every lattice in $G$ is
rigid.  On the other hand, it is known that generally lattices in
solvable Lie groups can be very far from being rigid.  In~\cite{St94}
Starkov gave important examples of rigid and non-rigid lattices in
solvable Lie groups.  In~\cite{Wi95} Witte proved that, if $G$ is
solvable, then every Zariski-dense lattice $\Gamma$ in $G$ is
superrigid in the following sense: any finite-dimensional
representation $\rho \colon \Gamma \to \GL_n(\R)$ virtually extends to
a representation of~$G$.  He also showed that, if $G$ is solvable and
$\Gamma$ is a Zariski-dense lattice in $G$, then every homomorphic
embedding of $\Gamma$ into $G$ as a lattice extends to a crossed
automorphism of $G$, i.e., a certain type of `twisted' automorphism.

In the present paper we initiate a quantitative description of the
phenomenon of non-rigidity for lattices in Lie groups.    Let
\begin{equation*} % \label{equ:def_X}
  \mathcal{X}(\Gamma,G) = \{ \phi \colon \Gamma \hookrightarrow G \mid
  \text{$\phi(\Gamma)$ is a lattice in $G$} \} 
\end{equation*}
be the space of all homomorphic embeddings of $\Gamma$ into $G$ as a
lattice, equipped with the topology of pointwise convergence.  The
space $\mathcal{X}(\Gamma,G)$ and its connected components feature in
classical works of Weil~\cite{We60} and Wang~\cite{Wa63}.  The group
$\Aut(G)$, consisting of all continuous automorphisms of $G$, is a Lie
group and acts continuously on $\mathcal{X}(\Gamma,G)$ from the left
via composition.  We observe that $\Gamma$ is rigid in $G$ if and only
if $\Aut(G)$ acts transitively on~$\mathcal{X}(\Gamma,G)$.  More
generally, the orbit space
$$ 
\mathcal{D}(\Gamma, G) = \Aut(G) \backslash \mathcal{X}(\Gamma,G),
$$ provides a quantitative measure for the degree of non-rigidity of
$\Gamma$ in~$G$.  It can be interpreted as the \emph{deformation
  space} of lattice embeddings of $\Gamma$ into~$G$.  We note that the
quotient space $\mathcal{D}(\Gamma, G)$ also reflects topological
properties of the $\Aut(G)$-orbits in~$\mathcal{X}(\Gamma,G)$.  

\medskip

Let
$G$ be a simply connected, solvable Lie group, and let $\Gamma$ be a
lattice in~$G$.
In general, deformation spaces of the form $\mathcal{D}(\Gamma, G)$
can be uncountable.  In the present article we are
particularly interested in describing principal situations where
$\mathcal{D}(\Gamma, G)$ is finite or countable.  
The group $G$ is said
to be \emph{unipotently connected} if its maximal nilpotent normal
subgroup is connected.  
Our main theorem
shows that Zariski-dense lattices in unipotently connected solvable Lie groups have
finite deformation spaces. 
%A key tool in our investigation is Mostow's
%algebraic hull construction for polycyclic groups and solvable Lie
%groups \cite{HoMo57, Mo70}. %; see Section \ref{sec:preliminaries}.

The following discussion
of our main results is built around three general themes: deformation
spaces of lattices in unipotently connected groups, the
characterisation of strong rigidity, and the topology of deformation
spaces.

\subsection*{Deformation spaces for unipotently connected groups} We prove the following finiteness result for 
deformation spaces: % Our main theorem is

\begin{thmABC} \label{thm:MainA} Let $G$ be a simply connected,
  solvable Lie group $G$ which is unipotently connected.  Then for
  every Zariski-dense lattice $\Gamma$ of $G$ the deformation space
  $\mathcal{D}(\Gamma, G)$ is finite. Moreover, its cardinality is
  uniformly bounded above by a constant depending only on the dimension of~$G$.
\end{thmABC}

We remark that, if $G$ is of real type, then $G$ is unipotently
connected and every lattice $\Gamma$ in $G$ is Zariski-dense.  Hence,
Theorem~\ref{thm:MainA} generalises the above mentioned rigidity
theorems of Mal'tsev and Sait\^o.  Furthermore, it yields the following
application, showing that $\Gamma$ is \emph{weakly rigid} in $G$ up to
finite index in~$\Aut(\Gamma)$.

\begin{corABC} \label{cor:MainB} Let $G$ be a simply connected,
  solvable Lie group $G$ which is unipotently connected, and let
  $\Gamma$ be a Zariski-dense lattice of~$G$.  Then the automorphisms
  of $\Gamma$ which extend to automorphisms of $G$ form a finite index
  subgroup of~$\Aut(\Gamma)$. The index of this subgroup is uniformly bounded above by a constant depending only on the rank of~$\Gamma$.

\end{corABC}

The constant alluded to in Theorem~\ref{thm:MainA} and
Corollary \ref{cor:MainB} is specified in
Section~\ref{sec:proof_of_A}.  \\
%Moreover, we show there that the
%cardinality of $\mathcal{D}(\Gamma, G)$ can also be bounded uniformly
%by a constant depending only on the dimension of the
%nilradical~$\Nil(G)$.  
%Corollary~\ref{cor:MainB} is established by
%showing that every element of the centraliser
%$\Cen_{\Aut(\Gamma)}(\Gamma/\Fitt(\Gamma))$ of the Fitting quotient
%of~$\Gamma$ extends to an automorphism of~$G$; see
%Section~\ref{sect:proofcorB}.

We emphasise that the conclusion of Theorem~\ref{thm:MainA} ceases to
hold if one of the assumptions, unipotent connectedness of the Lie
group $G$ or Zariski-denseness of the lattice $\Gamma$, is dropped.
In Section~\ref{sect:Examples} we construct a Zariski-dense lattice in
a non-unipotently connected Lie group $G$ which has a countably
infinite deformation space.  Moreover, we provide in the same section
examples of non-Zariski-dense lattices in unipotently connected groups
which have uncountable deformation spaces.  A construction of Starkov
shows that there exist rigid lattices which are not Zariski-dense;
see~\cite[Example~6.1]{St94}.

It is worth noting that unipotently connected Lie groups cover a wide
range of lattices.  Recall that every discrete subgroup of a solvable
Lie group is polycyclic; see~\cite{Mo57}.  By a result of Auslander,
every polycyclic group virtually embeds as a Zariski-dense lattice
into a suitable simply connected, solvable Lie group;
see~\cite{Au73,Wi95}.  We show that, in addition, the Lie group can be
taken to be unipotently connected; see
Proposition~\ref{pro:exists_unipotently_connected}.  This refinement
yields the following corollary.

\begin{corABC} \label{cor:toMainA1} Let $\Gamma$ be a polycyclic
  group.  Then $\Gamma$ admits a finite index subgroup $\Delta$ which
  embeds as a Zariski-dense lattice into a simply connected, solvable
  Lie group $H$ such that the deformation space $\mathcal{D}(\Delta,
  H)$ is finite.
\end{corABC}

The next corollary provides a simple structural criterion for a
polycyclic group $\Gamma$ to ensure that the deformation space
associated to any Zariski-dense lattice embedding of $\Gamma$ is
finite.
  
\begin{corABC} \label{cor:toMainA2} Let $\Gamma$ be a polycyclic group
  such that the commutator subgroup $[\Gamma, \Gamma]$ has finite
  index in the Fitting subgroup~$\Fitt(\Gamma)$.  Then for every
  Zariski-dense lattice embedding of $\Gamma$ into a simply connected,
  solvable Lie group $H$ the deformation space $\mathcal{D}(\Gamma,
  H)$ is finite.
\end{corABC} 

Indeed, the condition on $\Gamma$ in Corollary~\ref{cor:toMainA2}
implies that every Zariski-dense lattice embedding of $\Gamma$ is into
a unipotently connected Lie group; see
Proposition~\ref{pro:all_are_uc}.\\ 

The condition of unipotent connectedness will be further motivated and explained in Section~\ref{sec:uconnected}.  The class of unipotently
connected groups coincides with the class of groups (A) defined
in~\cite{St94}; the terminology we employ originates
from~\cite{GrSe94}.  If $\Gamma$ is a Zariski-dense lattice in $G$,
then $G$ is unipotently connected if and only if the dimension of the
nilradical of $G$ is equal to the rank of the Fitting subgroup
of~$\Gamma$; see Corollary~\ref{cor:nil_and_fitt}.  The Fitting
subgroup $\Fitt(\Gamma)$ is the maximal nilpotent normal subgroup
of~$\Gamma$. This illustrates that unipotently connected groups afford
strong `structural' links to their Zariski-dense lattices.

\subsection*{Strong rigidity and the structure set} Let $G$ be a
simply connected, solvable Lie group.  A Zariski-dense lattice
$\Gamma$ in $G$ is called \emph{strongly rigid} if every embedding of
$\Gamma$ as a Zariski-dense lattice into a simply connected, solvable
Lie group $H$ extends to an isomorphism of Lie groups $G \rightarrow
H$.  Zariski-denseness guarantees that such extensions are unique,
whenever they exist.  While rigidity of $\Gamma$ in $G$ is a property
which crucially depends on the ambient Lie group~$G$, strong rigidity
only depends on the group $\Gamma$ itself.  Indeed, we consider the
\emph{structure set} for Zariski-dense embeddings of $\Gamma$ into
simply connected, solvable Lie groups $H$, defined as
$$ 
\mathcal{S}^\textup{Z}(\Gamma) = \{ \phi\colon \Gamma \hookrightarrow
H \mid \phi(\Gamma) \text{ is a Zariski-dense lattice in $H$} \}
\big/ \sim \; , 
$$ where two embeddings $\phi_1 \colon \Gamma \hookrightarrow H_1$
and $\phi_2 \colon \Gamma \hookrightarrow H_2$ are equivalent if
there exists an isomorphism of Lie groups $\psi \colon H_1 \rightarrow H_2$
such that $\psi \circ \phi_1 = \phi_2$.  Clearly, $\Gamma$ is
strongly rigid if and only if $\mathcal{S}^\textup{Z}(\Gamma)$
consists of a single element.

For every simply connected, solvable Lie group $G$ such that $\Gamma$
is a Zariski-dense lattice in $G$, the deformation space
$\mathcal{D}(\Gamma,G)$ embeds naturally into the structure set
$\mathcal{S}^\textup{Z}(\Gamma)$; see
Section~\ref{sect:Structure_set}.  Our principal observation
concerning the structure set $\mathcal{S}^\textup{Z}(\Gamma)$ is the
following.

\begin{thmABC} \label{thm:MainE} Let $\Gamma$ be a Zariski-dense
  lattice in a simply connected, solvable Lie group~$G$.  Then the
  structure set $\mathcal{S}^\textup{Z}(\Gamma)$ is either countably
  infinite or it consists of a single element.  The structure set
  consists of a single element if and only if $\Gamma$ embeds as a
  lattice into a simply connected, solvable Lie group of real type.
\end{thmABC}

\begin{corABC} \label{cor:at_most_countable}
  Let $\Gamma$ be a Zariski-dense lattice in a simply connected,
  solvable Lie group~$G$.  Then $\mathcal{D}(\Gamma,G)$ is at most
  countably infinite.
\end{corABC}

In particular, Theorem \ref{thm:MainE} shows that $\Gamma$ is strongly
rigid if and only if $\Gamma$ is a lattice in a simply connected,
solvable Lie group of real type.  Combining Theorem \ref{thm:MainA}
and Theorem \ref{thm:MainE}, we prove the following dichotomy result
for Zariski-dense lattices in unipotently connected groups.

\begin{corABC} \label{cor:toMainE1} Let $\Gamma$ be a Zariski-dense
  lattice in a simply connected, solvable Lie group $G$ which is
  unipotently connected.  Then either $\Gamma$ is strongly rigid or
  there exist countably infinitely many pairwise non-isomorphic
  (unipotently connected) simply connected, solvable Lie groups which
  contain $\Gamma$ as a Zariski-dense lattice.
\end{corABC}

\subsection*{Topology of the deformation space} Let $G$ be a simply
connected, solvable Lie group, and let $\Gamma$ be a lattice in~$G$.
We apply our results to study the structure of the space
$\mathcal{X}(\Gamma,G)$ of lattice embeddings of $\Gamma$ into $G$
and, in particular, the properties of the $\Aut(G)$-action
on~$\mathcal{X}(\Gamma,G)$.  Recall that $\mathcal{X}(\Gamma,G)$ is a
subspace of the space $\Hom(\Gamma,G)$ of all homomorphisms of
$\Gamma$ into $G$, equipped with the topology of pointwise
convergence.  The group $\Aut(G)$, which naturally carries the
structure of a Lie group, acts continuously
on~$\mathcal{X}(\Gamma,G)$, and the deformation space
$$
\mathcal{D}(\Gamma, G) = \Aut(G) \backslash \mathcal{X}(\Gamma,G), 
$$
equipped with the quotient topology, reflects properties of the
$\Aut(G)$-orbits in~$\mathcal{X}(\Gamma,G)$.

Since $\Gamma$ is finitely generated, the space $\Hom(\Gamma,G)$
carries the additional structure of a real algebraic variety; in
particular, it has only finitely many connected components.  By a
celebrated result of Weil~\cite{We60}, for any cocompact lattice
$\Delta$ in a Lie group $H$ the space $\mathcal{X}(\Delta,H)$ is an
open and locally path connected subset of~$\Hom(\Delta,H)$.  In
classical situations, for example, in the case $H =
\mathrm{PSL}(2,\R)$, the space $\mathcal{X}(\Delta,H)$ coincides with
the union of two connected components of $\Hom(\Delta,H)$; see~\cite{FrKl,Go88}.
Moreover, the set $\mathcal{X}(\Delta,H)$ and its quotient, the
Teichm\"uller space, are connected manifolds which can be described by
algebraic equalities and inequalities; see~\cite{Ko97}.

For the simply connected, solvable Lie group $G$ a result of
Wang~\cite{Wa63} implies that the connected components of
$\mathcal{X}(\Gamma,G)$ are manifolds.  However, for a general
solvable Lie group $H$, even if $\Delta$ is a Zariski-dense lattice
in~$H$, the space $\mathcal{X}(\Delta,H)$ can have infinitely many
connected components; see Example~\ref{exm:nicht_u-zushg}.  This shows
that, in general, $\mathcal{X}(\Delta,H)$ cannot be described as a
semi-algebraic subset of~$\Hom(\Delta,H)$.

In contrast to the general picture, Theorem~\ref{thm:MainA} shows that
for a large class of lattices the space $\mathcal{X}(\Gamma,G)$ has
only finitely many connected components.  Indeed, whenever
$\mathcal{D}(\Gamma,G)$ is finite, the space $\mathcal{X}(\Gamma,G)$
has finitely many components, because the real algebraic group
$\Aut(G)$ has only finitely many components.

\begin{corABC} \label{cor:ToMainA3} Let $G$ be a simply connected
  solvable Lie group which is unipotently connected.  Then for every
  Zariski-dense lattice $\Gamma$ of $G$ the space of lattice
  embeddings $\mathcal{X}(\Gamma,G)$ has finitely many connected
  components.
\end{corABC}

The lattice $\Gamma$ is said to be \emph{locally rigid} in $G$ if the
$\Aut(G)$-orbit of the identity map $\phi_0 = \rm{id}_\Gamma \colon
\Gamma \hookrightarrow G$ is open in~$\mathcal{X}(\Gamma,G)$;
see~\cite[I, Chap.~1, \S 6.1]{OnVi00}.  The lattice $\Gamma$ is
\emph{deformation rigid} in $G$ if the identity component
$\Aut(G)_\circ$ acts transitively on the component of
$\mathcal{X}(\Gamma,G)$ containing~$\phi_0$; see~\cite[\S 7]{St94}.
In principle, $\Gamma$ could be locally rigid without being
deformation rigid.

The real algebraic hull $A_\Gamma$ of $\Gamma$ is a real algebraic
group, associated to $\Gamma$ in a functorial way, which we use to
control the Zariski-dense lattice embeddings of~$\Gamma$. Its
construction is originally due to Mostow~\cite{Mo70}.
% and has been applied successfully in the study of polycyclic groups
% and their interaction with Lie groups for example in ......
In Section~\ref{sect:Structure_set} we establish a one-to-one
correspondence between the structure set
$\mathcal{S}^\textup{Z}(\Gamma)$ and a certain collection
$\mathcal{G}(\Gamma)$ of closed Lie subgroups of~$A_\Gamma$.  As
explained in Section~\ref{sec:Chabauty}, this allows us to transfer
the natural topology on $\mathcal{G}(\Gamma)$ induced by the Chabauty
topology on the set of closed subgroups of~$A_\Gamma$
to~$\mathcal{S}^\textup{Z}(\Gamma)$.  Using Theorem~\ref{thm:MainE}
and the result of Wang~\cite{Wa63}, we prove the following.

\begin{thmABC} \label{thm:MainH} Let $\Gamma$ be a Zariski-dense
  lattice in a simply connected, solvable Lie group~$G$.  Then the
  structure set $\mathcal{S}^\textup{Z}(\Gamma)$ has the discrete
  topology and the embedding of $\mathcal{D}(\Gamma, G)$ into
  $\mathcal{S}^\textup{Z}(\Gamma)$ is continuous.
\end{thmABC}

In particular, every Zariski-dense lattice in a simply connected,
solvable Lie group is locally rigid.  In fact, from our proof of
Theorem~\ref{thm:MainH} we derive also the following
corollary, for which a proof sketch is provided 
in~\cite[Proposition~7.2 and its Corollary~7.3]{St94}.

\begin{corABC} \label{cor:toMainH1} Let $\Gamma$ be a Zariski-dense
  lattice in a simply connected, solvable Lie group~$G$.  Then $\Gamma$
  is deformation rigid in~$G$.
\end{corABC}

\medskip

We conclude the section with a brief description of the organisation
of the paper and a summary of the basic notation and terminology
employed.

\medskip

\noindent
\textit{Organisation.} In Section~\ref{sect:Examples} we present a
collection of instructive examples, illustrating our main results.  In
Section~\ref{sec:preliminaries} we summarise important facts about
lattices in solvable Lie groups and introduce tools with which to
study them.  In particular, we recall the algebraic hull construction
for polycyclic groups and solvable Lie groups; after this we discuss
the notion of tight Lie subgroups.  In Section~\ref{sect:Fitting} we
strengthen Mostow's theorem about the intersection of a lattice
$\Gamma$ with the nilradical of its ambient Lie group~$G$.
Section~\ref{sec:uconnected} is concerned with characterisations of
unipotently connected Lie groups.  In Section~\ref{sect:GGamma} we
introduce and describe the space $\mathcal{G}(\Gamma)$ which assists
in the description of Zariski-dense embeddings of a lattice~$\Gamma$.
Section~\ref{sect:GGammaG} is devoted to the space
$\mathcal{G}_{\Gamma,G}$ which gives a fibration for the deformation
space $\mathcal{D}(\Gamma,G)$ of a Zariski-dense lattice $\Gamma$
in~$G$.  Section~\ref{sec:proof_main_results} contains the proofs of
most of our main results, namely all results from
Theorem~\ref{thm:MainA} up to Corollary~\ref{cor:ToMainA3}.  Finally,
in Section~\ref{sec:Chabauty} we discuss the topologies on the
structure set $\mathcal{S}^\textup{Z}(\Gamma)$ and the deformation
space $\mathcal{D}(\Gamma,G)$, proving Theorem~\ref{thm:MainH} and its
Corollary~\ref{cor:toMainH1}.

\smallskip

\noindent
\paragraph{
\textit{Notation and terminology.}
} 
%In the present paper maps are
%applied from the left and actions are left actions.  In particular, we
%write $^{g\!}x = gxg^{-1}$ to denote left conjugation in a group.
%
%\smallskip

\begin{list}{}{\setlength{\labelwidth}{.2cm}
    \setlength{\leftmargin}{.4cm} \setlength{\itemsep}{.1cm}}
\item[Polycyclic groups.] A group $\Gamma$ is \emph{polycyclic} if
  there exists a finite series of subgroups $\Gamma = \Gamma_1
  \trianglerighteq \Gamma_2 \trianglerighteq \ldots \trianglerighteq
  \Gamma_{r+1} = 1$ such that each quotient~$\Gamma_i/\Gamma_{i+1}$ is
  cyclic.  The \emph{rank} $\rank(\Gamma)$ of a polycyclic group
  $\Gamma$ is the number of infinite cyclic factors in such a series,
  which is an invariant of~$\Gamma$. The \emph{Fitting subgroup}
  $\Fitt(\Gamma)$ of $\Gamma$ is the maximal nilpotent normal subgroup
  of~$\Gamma$.  See \cite{Se83} for further results on polycyclic
  groups.
\item[Linear algebraic groups] Let $\mathbf{G}$ be a linear algebraic
  group, defined over a field of characteristic~$0$. If $H$ is a
  subgroup of $\mathbf{G}$, we write $\ac{H}$ for the
  \emph{Zariski-closure} of $H$ in $\mathbf{G}$ and we denote by
  $\up(H)$ the \emph{collection of unipotent elements} in~$H$.  If $H$
  is solvable, $\up(H)$ is a subgroup. The \emph{identity component}
  of $\mathbf{G}$ is denoted by~$\mathbf{G}^\circ$.  The
  \emph{unipotent radical} of~$\mathbf{G}$, i.e., the maximal
  connected unipotent normal subgroup of $\mathbf{G}$, is denoted
  by~$\Rad_\textup{u}(\mathbf{G})$.  If $\mathbf{G}^\circ$ is a
  solvable group, then $\Rad_\textup{u}(\mathbf{G}) = \up(\mathbf{G})=
  \up(\mathbf{G}^\circ)$, and $[\mathbf{G}^\circ,\mathbf{G}^\circ]
  \subseteq \Rad_\textup{u}(\mathbf{G})$.

  The group $\mathbf{G}$ has a \emph{strong unipotent radical} if the
  centraliser of the unipotent radical
  $\Cen_{\mathbf{G}}(\Rad_\textup{u}(\mathbf{G}))$ is contained
  in~$\Rad_\textup{u}(\mathbf{G})$.  If $\mathbf{G}$ is defined over a
  field~$k$, then we denote by $\Aut_k(\mathbf{G})$ the group of
  $k$-defined automorphisms of~$\mathbf{G}$.  If $\mathbf{G}$ has a
  strong unipotent radical and is defined over $k$, then the group
  $\Aut_k(\mathbf{G})$ can be regarded as the group of $k$-points of a
  $k$-defined linear algebraic group; see~\cite{BaGr06}.
  See~\cite{Bo91} for a general reference on linear algebraic groups.

\item[Lie groups and lattices] In the present paper all Lie groups are
  real Lie groups.  Every simply connected Lie group is connected.

  A \emph{Lie subgroup} is an immersed submanifold which inherits a
  Lie group structure from the ambient group.  The \emph{identity
    component} of a Lie group $G$ is denoted by~$G_\circ$.  The
  \emph{nilradical} $\Nil(G)$ of a Lie group $G$ is the maximal
  connected nilpotent normal subgroup of~$G$.

  A real linear algebraic group $A$ is the group of $\R$-points $A=
  \mathbf{A}_{\R}$ of a linear algebraic group $\mathbf{A}$ defined
  over~$\R$.  In this case we write $A^\circ =
  \mathbf{A}^{\!\circ}_{\; \R}$.  Such a group is also a Lie group
  with respect to its natural Euclidean (Hausdorff) topology.  We have
  $A_{\circ} \leq A^\circ$ and $\lvert A : A_\circ \rvert < \infty$;
  see~\cite[Appendix]{Mo57} or, more generally,~\cite{Wh57}.

  A closed subgroup $H$ of a Lie group $G$ is said to be
  \emph{cocompact} (or uniform) if the quotient $G/H$ is compact.  A
  \emph{lattice} in $G$ is a discrete subgroup of finite co-volume,
  i.e.\ a discrete subgroup $\Gamma$ such that the space $G/\Gamma$
  admits a $G$-invariant probability measure.  If $G$ is a solvable
  Lie group, then a subgroup $\Gamma$ is a lattice in $G$ if and only
  if $\Gamma$ is discrete and cocompact.  Let $\mathfrak{g}$ be the
  Lie algebra associated to $G$, and let $\Ad \colon G \rightarrow
  \GL(\mathfrak{g})$ denote the adjoint representation of~$G$.  We say
  that a lattice $\Gamma$ is \emph{Zariski-dense} in $G$ if
  $\ac{\Ad(\Gamma)} = \ac{\Ad(G)}$ in the ambient real algebraic
  group~$\GL(\mathfrak{g})$.
  Every simply connected, nilpotent Lie group $G$ admits naturally the
  structure of a unipotent real linear algebraic group by using
  exponential coordinates.  With respect to this real algebraic
  structure, a subgroup $\Gamma$ is Zariski-dense in $G$ if and only
  if the closure of $\Gamma$ in the Euclidean topology of $G$ is a
  cocompact subgroup.  Every connected Lie subgroup of a unipotent
  real algebraic group is Zariski-closed.  See \cite{Ra72} for a
  general reference on lattices of Lie groups.
\end{list}

%%%%%

\section{Instructive examples} \label{sect:Examples}

The purpose of this section is to present a number of concrete
examples and useful constructions of lattices in solvable Lie groups.
While some of them are certainly well known, others highlight new
insights.  Many examples of rigid and non-rigid lattices in solvable
Lie groups can be found in~\cite{St94}.

In order to describe some explicit Lie groups and lattices we
parametrise $2$-by-$2$ rotation matrices over $\R$ by setting
\begin{equation}\label{equ:rotation_matrix}
  R(t) =
  \begin{pmatrix}
    \cos(2\pi t) & -\sin(2\pi t) \\
    \sin(2\pi y) & \phantom{-}\cos(2\pi t)
  \end{pmatrix},
\end{equation}
and we denote block diagonal matrices with blocks $B_1, \ldots, B_k$,
say, by~$\diag(B_1,\ldots,B_k)$.

%%%

\subsection{Uncountable deformation spaces}

We provide a simple example of a non-Zariski-dense lattice $\Gamma$ in
a non-unipotently connected group $G$ such that the deformation space
$\mathcal{D}(\Gamma,G)$ is uncountable.  We also explain that a
construction of Milovanov, discussed in~\cite[Example~2.9]{St94}, yields
an example of a non-Zariski-dense lattice $\Gamma$ in a unipotently
connected group $G$ such that $\mathcal{D}(\Gamma,G)$ is uncountable.
The latter shows that the assumption of Zariski-denseness of $\Gamma$
in Theorem~\ref{thm:MainA} is not redundant.

\begin{exm} \label{exm:E2+}
  Consider the group $\widetilde{E(2)^+}$, the universal cover
  of the group of orientation-preserving isometries of the Euclidean
  plane.  We realise an isomorphic copy of this group as follows:
  $$
  G = V . X(\R)  \cong \R^2 \rtimes \R,
  $$
  where $V = \R + \R i = \C$ and the one-parameter group $X(t)$ acts
  as multiplication by $e^{2\pi t i}$ on~$V$.  Then $\Gamma = (\Z + \Z
  i) . X(\Z) \cong \Z^3$ is a lattice in~$G$ which is not
  Zariski-dense.  We claim that the deformation space
  $\mathcal{D}(\Gamma,G)$ is uncountable.

  \smallskip

  We observe that $\Zen(G) = X(\Z)$ and $\Nil(G) = V$; in particular,
  $G$ is not unipotently connected.  If $\Lambda$ is any lattice in
  the vector group~$V$, then $\Lambda X(\Z)$ is a lattice in $G$ which
  is isomorphic to~$\Gamma$.  Now $V$ is characteristic in~$G$.  Hence
  $\Aut(G)$ acts on the collection of all lattices in~$V$, and
  $\mathcal{D}(\Gamma,G)$ maps onto
  $$
  \Aut(G) \backslash \{ \Lambda \mid \Lambda \text{ a lattice in } V
  \}.
  $$
  Consequently, it suffices to show that the latter space is
  uncountable.  For our purposes it is more convenient to work with
  the identity component $\Aut(G)_\circ$ of $\Aut(G)$; this is
  permissible, because $\Aut(G)_\circ$ has finite index in~$\Aut(G)$.
  Two elements of $V$ are conjugate in $G$ if and only if they have
  the same modulus (as complex numbers).  This shows that any
  automorphism of $G$ acts on $V$ as multiplication by a non-zero
  complex number, possibly followed by complex conjugation.  It is
  easy to write down automorphisms of $G$ which induce this action,
  and we deduce that the action of $\Aut(G)_\circ$ on $V$ is
  equivalent to the action of $\C^*$ on $\C$ by multiplication.  The
  situation is now classical: the space of lattices in $\C$ up to the
  action of $\C^*$ is isomorphic to the upper half plane modulo the
  action of $\textup{SL}_2(\Z)$, hence uncountable.
\end{exm}

\begin{exm} \label{exm:Milovanov} Based on an example of Milovanov and
  Starkov ~\cite[Example~2.9]{St94}, we construct a non-Zariski-dense lattice $\Gamma$ in a
 unipotently connected group $G$ such that $\mathcal{D}(\Gamma,G)$ is
  uncountable.

  \smallskip  

  Let $A = \diag(A_0,A_0) \in \GL_4(\Z)$, where $A_0 \in \GL_2(\Z)$ is
  the companion matrix of the polynomial $f = x^2 - 3x + 1$.  Observe
  that $f$ splits over $\R$ into a product of two distinct linear
  factors $x-\lambda$ and $x-\lambda^{-1}$, say.
  
  We define in $\GL_4(\R)$ the one-parameter subgroup
  $$
  X(t) = \diag(\lambda^t R(t),\lambda^{-t},\lambda^{-t}) \qquad (t \in
  \R) .
  $$
%  where we use the notation introduced in~\eqref{equ:rotation_matrix}.
  We view $A$ as an operator on a $4$\nobreakdash-dimensional vector
  space $V$ over $\R$ with basis $v_1, v_2, v_3, v_4$, say.  Thus
  $\Lambda = \oplus_{j=1}^4 \Z v_j$ is an $A$-invariant full
  $\Z$-lattice in~$V$.  Then $V$ decomposes into a direct sum $V =
  V^\lambda \oplus V^{\lambda^{-1}}$ of $A$-invariant $2$-dimensional
  subspaces corresponding to the eigenvalues of~$A$.  Choosing bases
  for these subspaces, we obtain a new basis for $V$ so that the
  action of $A$ on $V$ with respect to this basis is given by~$X(1)$.
  Then $\Gamma = \Lambda \rtimes X(\Z)$ is a lattice of the simply
  connected, solvable Lie group $G = V \rtimes X(\R)$.  Observe that
  the maximal nilpotent normal subgroup of $G$ is equal to $\Nil(G) =
  V$ so that $G$ is unipotently connected.  Furthermore,
  %  it is shown in~\cite[Example~2.9]{St94} that 
   the lattice $\Gamma$ is not Zariski-dense in~$G$.

  In order to show that $\mathcal{D}(\Gamma,G)$ is uncountable we
  argue similarly as in Example~\ref{exm:E2+}.  We consider the subspace of lattice
  embeddings of $\Gamma$ into $G$ which map $X(1)$ to itself.  Such
  embeddings are in one-to-one correspondence with $X(\Z)$-equivariant
  embeddings of $\Lambda$ into~$V$, or in other words with elements of
  the centraliser of $X(\Z)$ in~$\GL_4(\R)$.  This centraliser is
  isomorphic to $\GL_2(\R) \times \GL_2(\R)$, preserving the subspaces
  $V^\lambda$ and $V^{\lambda^{-1}}$.  On the other hand we consider
  lattice embeddings of $\Gamma$ into $G$ mapping $X(1)$ to itself
  which are induced by automorphisms of~$G$.  These correspond to
  embeddings of $\Lambda$ into $V$ which are $X(\R)$-equivariant, or
  in other words to elements of the centraliser of $X(\R)$
  in~$\GL_4(\R)$.  They, too, preserve the subspaces $V^\lambda$
  and~$V^{\lambda^{-1}}$, but form a strictly smaller group isomorphic
  to $\C^* \times \GL_2(\R)$.
%   where $\textup{GO}_2(\R)$
%  denotes the orthogonal similitude group associated to the standard
%  symmetric bilinear form.  
  Similarly as in Example~\ref{exm:E2+}, we
  conclude that a subset of $\mathcal{D}(\Gamma,G)$ maps onto
  $\C^* \backslash \GL_2(\R)$.  The latter space is
  homeomorphic to the upper half plane.  Thus $\mathcal{D}(\Gamma,G)$
  is uncountable.
\end{exm}

%%%

\subsection{Finite deformation spaces}
In view of Theorem~\ref{thm:MainA} it is natural to try to construct
Zariski-dense lattices in unipotently connected groups yielding
deformation spaces of various finite cardinalities.  We show that
deformation spaces of arbitrarily large sizes can be obtained.

\begin{exm} \label{exm:factorial} Let $n \in \N_0$.  Based on examples
  of Auslander and Starkov, we construct a Zariski-dense lattice
  $\Gamma$ in a simply connected, solvable Lie group $G$ such that the
  corresponding deformation space $\mathcal{D}(\Gamma,G)$ is finite of
  size~$(n+1)!$.
 
  \smallskip

  Let $A \in \GL_4(\Z)$ be the companion matrix of the polynomial $f =
  x^4 - 8x^3 + 10x^2 -8x + 1$.  It is easy to check that $f$ splits
  over $\R$ into a product of two linear factors and one irreducible
  quadratic factor.  Indeed, the eigenvalues of $A$ are
  $$
  \lambda, \quad \lambda^{-1}, \quad \alpha = e^{2\pi i \theta}, \quad
  \overline{\alpha} = e^{-2\pi i \theta},
  $$
  where $\lambda \in \R$ satisfies $\lambda > 1$ and $\theta \in
  (0,1/2)$ is irrational.

  Let $k \in \{0,\ldots,n\}$.  We define in $\GL_4(\R)$ the
  one-parameter subgroup
  $$
  X_k(t) = \diag(\lambda^t,\lambda^{-t}, R(t(\theta + k))
  \qquad (t \in \R),
  $$
  where we use the notation introduced in~\eqref{equ:rotation_matrix}.
  We view $A$ as an operator on a $4$-dimensional vector space $V_k$
  over $\R$ with basis $v_{k,1}, v_{k,2}, v_{k,3}, v_{k,4}$, say.
  Thus $\Lambda_k = \oplus_{j=1}^4 \Z v_{k,j}$ is an $A$-invariant
  full $\Z$-lattice in~$V_k$.  Then $V_k$ decomposes into a direct sum
  $V_k = V_k^\lambda \oplus V_k^{\lambda^{-1}} \oplus
  V_k^{\alpha,\alpha^{-1}}$ of $A$-invariant subspaces corresponding
  to the eigenvalues of~$A$.  Choosing appropriate bases for these
  subspaces, we obtain a new basis for $V_k$ so that the action of $A$
  on $V_k$ with respect to this basis is given by~$X_k(1)$.  Then
  $\Gamma_k = \Lambda_k \rtimes X_k(\Z)$ is a lattice of the simply
  connected, solvable Lie group $G_k = V_k \rtimes X_k(\R)$.
  According to~\cite[Example~4.1]{St94}, the lattice $\Gamma_k$ is
  Zariski-dense and rigid in~$G_k$: every lattice embedding of
  $\Gamma_k$ into $G_k$ extends to an automorphism of~$G_k$.

  Next we define
  $$
  G = \prod_{j=0}^n G_j \qquad \text{and} \qquad \Gamma =
  \prod_{j=0}^n \Gamma_j.
  $$
  Clearly, $\Gamma$ is a Zariski-dense lattice of the simply
  connected, solvable Lie group~$G$.  Whereas the Lie groups $G_0$,
  \ldots, $G_n$ are mutually non-isomorphic, their lattices
  $\Gamma_0$, \ldots, $\Gamma_n$ are all isomorphic to one another.
  Thus we obtain an embedding 
  $$
  \textup{Sym}(n+1) \hookrightarrow \mathcal{X}(\Gamma,G), \quad
  \sigma \mapsto \phi_\sigma
  $$
  as follows.  Viewing the symmetric group $\textup{Sym}(n+1)$ as a
  permutation group of $\{0,\ldots,n\}$, acting from the right, the
  image $\phi_\sigma \in \mathcal{X}(\Gamma,G)$ of $\sigma \in
  \textup{Sym}(n+1)$ is defined as
  $$
  \phi_\sigma \colon \Gamma \hookrightarrow G, \quad (g_j)_{j=0}^n
  \mapsto \left( \iota_{j\sigma,j} (g_{j\sigma})
  \right)_{j=0}^n,
  $$
  where $\iota_{j\sigma,j} \colon \Gamma_{j\sigma} \rightarrow
  \Gamma_j$ denotes, for each $j \in \{0,\ldots,n\}$, the isomorphism
  determined by $\iota_{j\sigma,j}(X_{j\sigma}(1)) = X_j(1)$ and
  $\iota_{j\sigma,j}(v_{j\sigma}^m) = v_j^m$ for $m \in \{1,2,3,4\}$.
  We now claim that
  \begin{equation}\label{equ:AutG}
    \Aut(G) = \prod_{j=0}^n \Aut(G_j)
  \end{equation}
  and that
  \begin{equation}\label{equ:deformSym}
    \mathcal{D}(\Gamma,G) = \{ [\phi_\sigma] \mid \sigma \in
    \textup{Sym}(n+1) \},
  \end{equation}
  implying that the deformation space has size $(n+1)!$, as wanted.

  In order to prove~\eqref{equ:AutG}, consider $\psi \in \Aut(G)$.  We
  note that $V = \prod_{j=0}^n V_j = \Nil(G)$ is mapped to itself
  by~$\psi$.  Furthermore, the factors $V_j$ in this product are
  characterised as the $G$-invariant subgroups of $V$ which are
  maximal with the property that their centralisers in $G$ have
  codimension~$1$.  Thus they are permuted among themselves by~$\psi$.
  Moreover, the action of $G$ on each group $V_j$ by conjugation
  factors through $X_j(\R)$ and thus defines~$G_j$.  As these Lie
  groups are mutually non-isomorphic, it is clear that $\psi$ maps
  each $V_j$ isomorphically onto itself.  Now fix $k \in
  \{0,\ldots,n\}$.  We have $V X_k(\R) = \Cen_G(\oplus_{j \ne k} V_j)$
  and hence $V \psi(X_k(\R)) = V X_k(\R)$.  Since $\psi(X_k(\R))$ and
  $\psi(X_j(\R))$ are to commute for any $j \in \{0,\ldots,n\}$, we
  conclude that $\psi(X_k(\R)) \subseteq V_k X_k(\R) = G_k$ and hence
  $\psi(G_k) = G_k$.  This finishes the proof of~\eqref{equ:AutG}.

  It remains to justify the inclusion~`$\subseteq$'
  in~\eqref{equ:deformSym}.  Let $\phi \colon \Gamma \hookrightarrow
  G$ be a lattice embedding, and observe that $\phi(\Gamma)$ is also
  Zariski-dense in~$G$; see Corollary~\ref{cor:all_Zdense}.  Set
  $\Lambda = \prod_{j=0}^n \Lambda_j$, the Fitting subgroup
  of~$\Gamma$.  By a theorem of Mostow~\cite[\S5]{Mo57}, the
  intersection $\phi(\Gamma) \cap V$ is a lattice of $V \cong
  \R^{4(n+1)}$ and thus $\phi(\Gamma) V/V$ is a lattice of $G/V \cong
  \R^n$.  Since $V$ is the maximal nilpotent normal subgroup of $G$,
  we must have $\phi(\Gamma) \cap V = \phi(\Lambda)$.  Fix $k \in
  \{0,\ldots,n\}$.  The Zariski-closure $\ac{\phi(\Lambda_k)}$ of
  $\phi(\Lambda_k)$ in $V$ has dimension~$4$, and since the
  centraliser of $\Lambda_k$ in $\Gamma$ has co-rank~$1$, the
  centraliser of $\ac{\phi(\Lambda_k)}$ in $G$ has co-dimension~$1$.
  As observed earlier, this implies that $\ac{\phi(\Lambda_k)} =
  V_{k\sigma}$, where $\sigma \in \textup{Sym}(n+1)$ is a permutation
  of $\{0,\ldots,n\}$.  Arguing similarly as before, we deduce from
  $\Lambda X_k(\Z) = \Cen_\Gamma(\oplus_{j \ne k} \Lambda_j)$ that
  $$
  \ac{V \phi(X_k(\Z))} = \Cen_G(\oplus_{j \ne k} \ac{\phi(\Lambda_j)})
  = \Cen_G(\oplus_{j \ne k} V_{j\sigma}) = V X_{k\sigma}(\R).
  $$
  Since $\phi(X_k(\Z))$ and $\phi(X_j(\Z))$ are to commute for any $j
  \in \{0,\ldots,n\}$, we conclude that $\phi(X_k(\Z)) \subseteq
  V_{k\sigma} X_{k\sigma}(\R) = G_{k\sigma}$, and hence
  $\phi(\Gamma_k)$ is a lattice in~$G_{k\sigma}$.  Since
  $\phi(\Gamma_k) \cong \Gamma_{k\sigma}$ and since $\Gamma_{k\sigma}
  = \phi_{\sigma^{-1}}(\Gamma_k)$ is rigid in $G_{k\sigma}$ we find
  $\gamma_{k\sigma} \in \Aut(G_{k\sigma})$ such that $\gamma_{k\sigma}
  \circ \phi \vert_{\Gamma_k} = \phi_{\sigma^{-1}} \vert_{\Gamma_k}$.
  In view of~\eqref{equ:AutG}, this proves~\eqref{equ:deformSym}.
\end{exm}

%%%

\subsection{Countably infinite deformation spaces}
Corollary~\ref{cor:at_most_countable} states that the deformation
space of a Zariski-dense lattice in a simply connected, solvable Lie
group $G$ is at most countably infinite.  We construct a Zariski-dense
lattice in a non-unipotently connected group which has an infinite
deformation space.  The example illustrates that the assumption of
unipotent connectedness of $G$ in Theorem~\ref{thm:MainA} is not
redundant.

\begin{lem} \label{lem:polynomials} There are infinitely many
  polynomials $f \in \Z[x]$ which factorise over $\C$ as
  $$
  f = (x - \alpha) (x - \overline{\alpha}) (x - \alpha^{-1}) ( x -
  \overline{\alpha}^{-1}),
  $$
  where $\alpha = \lambda e^{2\pi i \theta}$ with $\lambda \in
  \R_{>1}$ and $\theta \in (0,1/2)$ irrational.  We can further
  arrange that the polynomials are irreducible over~$\Q$.
\end{lem}

\begin{proof}
  We consider the collection $\mathcal{F} \subseteq \Z[x]$ of all
  polynomials
  $$ 
  x^4 - a x^3 + b x^2 - ax + 1
  $$
  with $a,b \in \N$ such that $a,b \equiv_2 1$ and $b >
  \max\{a^2/2,2a+2\}$.

  Since $x^4 + x^3 + x^2 + x + 1$ is irreducible over the field
  $\mathbb{F}_2$ of cardinality~$2$, we conclude that all polynomials
  in $\mathcal{F}$ are irreducible over $\Z$, and hence over~$\Q$.
  
  Now consider $f = x^4 - a x^3 + b x^2 - ax + 1 \in \mathcal{F}$.
  Two short computations show that
  \begin{itemize}
  \item[(i)] $f$ has no real roots, because $b > a^2/2$.
  \item[(ii)] $f$ has no complex roots of modulus~$1$, because $b >
    2a+2$.
  \end{itemize}
  Hence, $f = (x - \alpha) (x - \overline{\alpha}) (x - \alpha^{-1}) (
  x - \overline{\alpha}^{-1} )$, where $\alpha = \lambda e^{2\pi i
    \theta}$ with $\lambda \in \R_{>1}$ and $\theta \in (0,1/2)$.
  From $a = \alpha + \overline{\alpha} + \alpha^{-1} +
  \overline{\alpha}^{-1} = (\lambda + \lambda^{-1}) 2
  \cos(2\pi\theta)$ we see that $a$ and $\theta$ determine $\lambda$,
  and hence~$b$.  Consequently, if we fix $a$, then different values
  for $b$ correspond to different values for~$\theta$.  On the other
  hand, the degree of the field $\mathbb{Q}(\alpha,\overline{\alpha})$
  over $\mathbb{Q}$ is uniformly bounded by $8$ and thus
  $\mathbb{Q}(\alpha,\overline{\alpha})$ contains roots of unity only
  up to a certain, uniformly bounded order.  Since $e^{4\pi i \theta}
  = \alpha \overline{\alpha}^{-1} \in
  \mathbb{Q}(\alpha,\overline{\alpha})$, this means that, for fixed
  $a$, for almost all $b$, the resulting value for $\theta$ must be
  irrational.

  Hence, for fixed $a$, for almost all $b$ the polynomial $f = x^4 - a
  x^3 + b x^2 - ax + 1 \in \mathcal{F}$ has the required properties
\end{proof}

\begin{exm} \label{exm:nicht_u-zushg} We construct a Zariski-dense
  lattice $\Gamma$ in a simply connected, solvable Lie group $G$ such
  that $\Gamma / \Fitt(\Gamma)$ is not torsion-free.  Consequently,
  the Lie group $G$ is not unipotently connected; see
  Lemma~\ref{lem:fitt_quo_tor_free}.  We go on to show that the
  deformation space $\mathcal{D}(\Gamma,G)$ is infinite.
  
  \smallskip
  
  Let $A \in \GL_4(\Z)$ with complex eigenvalues
  $$
  \alpha = \lambda e^{2\pi i \theta}, \quad \overline{\alpha} =
  \lambda e^{-2\pi i \theta}, \quad \overline{\alpha}^{-1} =
  \lambda^{-1} e^{2\pi i \theta}, \quad \alpha^{-1} = \lambda^{-1}
  e^{-2\pi i \theta},
  $$
  where $\lambda \in \R$ satisfies $\lambda > 1$ and $\theta \in
  (0,1/2)$ is irrational.  For instance, one can take the companion
  matrix of one of the polynomials in Lemma~\ref{lem:polynomials}.

  Alternatively, one can take for $A$ the companion matrix of the
  concrete polynomial $f = x^4 + 22x^3 + 150x^2 + 22x + 1$.  One
  checks that $A = B^3$ for $B \in \GL_4(\R)$ with characteristic
  polynomial
  $$
  x^4 + 4x^3 + 3x^2 -2x + 1 = (x+1)^4 - 3(x+1)^2 + 3.
  $$
  This polynomial appears in a construction of Wilking; see
  \cite[Example~2.1]{Wi00}.  He showed that the eigenvalues of $B$,
  which one computes easily, have angular component $2\pi i \theta$
  with irrational~$\theta$.  Thus $A$ has the desired property.

  We define in $\GL_5(\R)$ the one-parameter subgroups
  \begin{align*}
    X(t) & = \diag(\lambda^t R(t\theta),\lambda^{-t} R(-t \theta),1)
    && (t \in \R), \\
    Y(t) & = \diag(2^t R(t/2),2^t R(-t/2),2^t) && (t \in \R),
  \end{align*}
  where we employ the notation introduced in
  \eqref{equ:rotation_matrix}.  Clearly, $X(\R)$ and $Y(\R)$ commute
  with one another.

  We view $A$ as an operator on a $4$-dimensional vector space $V$
  over $\R$ with basis $v_1, v_2, v_3, v_4$, say.  Thus $\Lambda =
  \oplus_{i=1}^4 \Z v_i$ is an $A$-invariant full $\Z$-lattice
  in~$V$. Then $V$ decomposes into a direct sum $V = V_{\alpha,
    \overline{\alpha}} \oplus V_{\alpha^{-1}, \overline{\alpha}^{-1}}$
  of $A$-invariant planes corresponding to the eigenvalue pairs
  $\alpha, \overline{\alpha}$ and $\alpha^{-1},
  \overline{\alpha}^{-1}$.  Choosing appropriate bases for $V_{\alpha,
    \overline{\alpha}}$ and $V_{\alpha^{-1}, \overline{\alpha}^{-1}}$,
  we obtain a new basis $e_1,e_2,e_3,e_4$ so that, if the abelian
  group $V$ is embedded into $\GL_5(\R)$ via
  $$
  \eta \colon V \rightarrow \GL_5(\R), \qquad \sum_{i=1}^4 x_i e_i
  \mapsto
  \begin{pmatrix}
    \Id & \underline{x} \\
    0 & 1
  \end{pmatrix},
  $$
  where $\underline{x} = (x_1,x_2,x_3,x_4)^\text{t}$, then the
  original action of $A$ on $V$ is isomorphic to that of $X(1)$ on
  $\eta(V)$ by conjugation in~$\GL_5(\R)$.

  Now we consider the simply connected, solvable Lie group
  $$
  G = \eta(V) . X(\R) Y(\R) \cong \R^4 \rtimes (\R \times \R)
  $$ 
  and its lattice
  $$
  \Gamma = \eta(\Lambda) . X(\Z) Y(\Z) \cong \Z^4 \rtimes (\Z \times
  \Z).
  $$
  It is easily seen that $\Zen(G) = Y(2\Z)$ and that $V Y(2\Z)$ is the
  maximal nilpotent normal subgroup of~$G$.  Thus $\Nil(G) = V$ and
  $G$ is not unipotently connected.  Furthermore, we have
  $$
  \Fitt(\Gamma) = \eta(\Lambda) . Y(2\Z) \cong \Z^5 \qquad \text{and}
  \qquad \Gamma/\Fitt(\Gamma) \cong \Z \times \Z/2\Z.
  $$

  Next we prove that $\Gamma$ is Zariski-dense in~$G$.  For this it
  suffices to show that $\Gamma$ and $G$ have the same Zariski-closure
  in~$\GL_5(\R)$.  Clearly, $\eta(V)$ is Zariski-closed and
  $\ac{\eta(\Lambda)} = \ac{\eta(V)} = \eta(V)$.  The group $X(\R)$ is
  contained in the $2$-dimensional real algebraic torus
  $$
  T = T_\text{s} \times T_\text{c},
  $$
  where 
  \begin{align*}
    T_\text{s} & = \{ \diag(r, r, r^{-1},
    r^{-1},1) \mid r \in \R^* \} \cong \GL_1(\R), \\
    T_\text{c} & = \{ \diag(R(r),R(-r),1) \mid r \in \R \} \cong
    \textup{SO}_2(\R).
  \end{align*}
  From the eigenvalues of the elements in $X(\Z) \subseteq X(\R)$ we
  see that $X(\Z)$ is neither contained in a split real algebraic
  torus nor in a compact real algebraic torus.  This means that the
  algebraic closure of $X(\Z)$ is at least $2$-dimensional and hence
  $\ac{X(\Z)} = \ac{X(\R)} = T$.  Finally, we note that the group
  $Y(\R)$ is contained in the $2$-dimensional real algebraic torus
  $$
  S = S_\text{s} \times S_\text{c},
  $$
  where
  $$
  S_\text{s} = \{ \diag(r,r,r,r,r) \mid r \in \R^* \} \cong \GL_1(\R)
  \quad \text{and} \quad S_\text{c} = T_\text{c}.
  $$
  Again from the eigenvalues of elements in $Y(\R)$ it follows that
  $\ac{Y(\R)} = S$, and a small computation shows that 
  $$
  \ac{Y(\Z)} = S_\text{s} \times \langle \diag(-1,-1,-1,-1,1) \rangle.
  $$
  Altogether we see that 
  $$
  \ac{\Gamma} = \eta(V) (T_\text{s} \times S_\text{s} \times
  T_\text{c}) = \ac{G}.
  $$

  It remains to be shown that $\mathcal{D}(\Gamma,G)$ is infinite, and
  hence countably infinite by Corollary~\ref{cor:at_most_countable}.
  The automorphism group $\Aut(G)$ acts on $G/\Nil(G) = G/V$ as a
  finite group of automorphisms; see Lemma \ref{lem:autoinaut1}.
  On the other hand
  $$
  \Gamma_m = \eta(\Lambda) . X(\Z) Y((2m+1)\Z), \quad m \in \N,
  $$
  gives a countably infinite family of lattices in $G$, which are
  distinct modulo $V$ and each of which is isomorphic to~$\Gamma$.
  Thus $\mathcal{D}(\Gamma,G)$ is infinite.
\end{exm}

%%%

\subsection{Unipotently connected groups and their lattices}

The construction in Example~\ref{exm:nicht_u-zushg} also provides
interesting examples of lattices in unipotently connected groups.  We
comment on this matter just after
Lemma~\ref{lem:fitt_quo_tor_free}.
%%%%%

\section{Algebraic hulls, density properties of lattices and tight Lie
  subgroups} \label{sec:preliminaries} In this section we develop
important facts about lattices in solvable Lie groups.  In particular,
we introduce the algebraic hull construction, our central tool in the
study of such lattices.  Most of the material is implicit in the works
of Mostow, in particular \cite{Mo54,Mo57,Mo70}, and Auslander
~\cite{Au73}; see also \cite{Ra72} and~\cite{St94}.  In the last part
of the section we define and parametrise tight Lie subgroups of a
solvable real algebraic group.

\subsection{Algebraic hulls} One of our key tools is the algebraic
hull construction for polycyclic groups and solvable Lie groups which
is originally due to Mostow \cite{Mo70}.  In the context of polycyclic
groups, the algebraic hull can be regarded as a generalisation of the
Mal'tsev completion of a finitely generated, torsion-free nilpotent
group.  The construction is related to the notion of semisimple
splittings, which originates in the works of Mal'tsev, Wang and
Auslander on solvmanifolds.  We recall some of the key features of the
algebraic hull construction.  For further details see
\cite[Chap.~IV]{Ra72}, \cite[App.~A]{Ba04} and
\cite[Chap.~1]{Ba08}.

\subsubsection{Polycyclic groups} Let $\Gamma$ be a polycyclic group.
It is known that $\Cen_\Gamma(\Fitt(\Gamma)) \subseteq \Fitt(\Gamma)$;
see~\cite[\S{}2B]{Se83}.  Moreover, the following conditions are
equivalent:
\begin{itemize}
\item $\Fitt(\Gamma)$ is torsion-free;
\item besides the trivial group, $\Gamma$ has no finite normal
  subgroups.
\end{itemize}
Suppose that one of these conditions is satisfied.  Then there exist a
$\Q$-defined linear algebraic group $\mathbf{A}$ and an embedding
$\iota \colon \Gamma \hookrightarrow \mathbf{A}$ such that
$\iota(\Gamma) \subseteq \mathbf{A}_\Q$ and
\begin{itemize}
\item[(i)] $\iota(\Gamma)$ is Zariski-dense in $\mathbf{A}$,
\item[(ii)] $\mathbf{A}$ has a strong unipotent radical, i.e.,
  $\Cen_{\mathbf{A}}(\Rad_\textup{u}(\mathbf{A})) \subseteq
  \Rad_\textup{u}(\mathbf{A})$,
\item[(iii)] $\dim \Rad_\textup{u}(\mathbf{A}) = \rank \Gamma$.
\end{itemize}
Moreover, the construction $\iota \colon \Gamma \hookrightarrow
\mathbf{A}$ is uniquely determined up to $\Q$-isomorphism of linear
algebraic groups; see Corollary~\ref{cor:hull_unique}.  We thus refer
to $\mathbf{A}$, together with a possibly implicit embedding $\iota$,
as the \emph{algebraic hull} of~$\Gamma$.  In dealing with several
groups at the same time, it will be convenient to denote the algebraic
hull of $\Gamma$ by $\mathbf{A}_{\Gamma}$ and to assume that $\iota$
is simply the inclusion map: $\Gamma \subseteq \mathbf{A}_\Gamma$.

Now let $k$ be a field of characteristic~$0$.  Then a
\emph{$k$-defined algebraic hull} for $\Gamma$ is a $k$-defined linear
algebraic group $\mathbf{A}$ together with an embedding $\iota \colon
\Gamma \hookrightarrow \mathbf{A}$ such that $\iota(\Gamma) \subseteq
\mathbf{A}_k$ and satisfying conditions (i), (ii), (iii) above.  The
group of $k$-points $A = \mathbf{A}_k$, together with the embedding
$\iota \colon \Gamma \hookrightarrow A$, is called a $k$-algebraic
hull of~$\Gamma$.  In the special case $k = \R$, we call $A$ the
\emph{real algebraic hull} of~$\Gamma$.  If $k$ is an algebraic number
field with ring of integers $\mathcal{O}$, then the $k$-defined
algebraic hull $\mathbf{A}$ of $\Gamma$ has the additional property
that
\begin{itemize}
\item[(iv)] for any $k$-defined representation $\rho \colon \mathbf{A}
  \rightarrow \GL_n(\C)$, $n \in \N$, the pre-image $\iota^{-1}
  \rho^{-1}(\GL_n(\mathcal{O}))$ has finite index in~$\Gamma$.
\end{itemize}

We remark that the algebraic hull $\mathbf{A}_\Gamma$ of a polycyclic
group $\Gamma$ need not be connected, even if $\Gamma$ is
poly-(infinite cyclic).  For instance, if $\Gamma \cong \Z \rtimes \Z$
is the fundamental group of the Klein bottle, then
% the algebraic hull of $\Gamma$ is $(\mathbf{G}_{\text{a}} \times
% \mathbf{G}_{\text{a}}) \rtimes \{1,-1\}$ and
the real algebraic hull of $\Gamma$ is isomorphic to $\R^2 \rtimes
\{1,-1\}$.

The key property of the algebraic hull construction is recorded in the
following lemma.

\begin{lem}[Extension Lemma] \label{lem:extension} Let $\Gamma$ be a
  polycyclic group. Let $k$ be a field of characteristic $0$, and
  suppose that $\mathbf{A}$ is a $k$-defined algebraic hull
  for~$\Gamma$.  Let $\mathbf{B}$ be a linear algebraic $k$-group
  which has a strong unipotent radical.  Then every homomorphism $\phi
  \colon \Gamma \rightarrow \mathbf{B}$ with Zariski-dense image
  $\phi(\Gamma) \subseteq \mathbf{B}$ extends uniquely to a
  $k$-defined morphism of algebraic groups $\Phi \colon \mathbf{A}
  \rightarrow \mathbf{B}$.
\end{lem}

\begin{proof}
  A variant of this lemma is proved in \cite[Proposition~1.4]{Ba08};
  compare also \cite[Lemma~4.41]{Ra72}.
\end{proof}

\begin{cor} \label{cor:hull_unique} Let $\Gamma$ be a polycyclic
  group.  Let $k$ be a field of characteristic $0$, and suppose that
  $\mathbf{A}$, $\mathbf{B}$ are $k$-defined algebraic hulls
  for~$\Gamma$.  Then the identity map on $\Gamma$ extends to a
  $k$-defined isomorphism of algebraic groups $\Phi \colon \mathbf{A}
  \rightarrow \mathbf{B}$.
\end{cor}

\begin{proof}
  By Lemma~\ref{lem:extension}, the identity map on $\Gamma$ extends
  to $k$-defined morphisms of algebraic groups $\Phi \colon \mathbf{A}
  \rightarrow \mathbf{B}$ and $\Psi \colon \mathbf{B} \rightarrow
  \mathbf{A}$.  Note that $\Psi \circ \Phi \colon \mathbf{A}
  \rightarrow \mathbf{A}$ restricts to the identity on~$\Gamma$.
  Since $\Gamma$ is Zariski-dense in~$\mathbf{A}$, we conclude that
  $\Psi \circ \Phi = \id_\mathbf{A}$.  Similarly, one shows that $\Phi
  \circ \Psi = \id_\mathbf{B}$.  Hence $\Phi$ and $\Psi$ are mutual
  inverses of each other, and $\Phi \colon \mathbf{A} \rightarrow
  \mathbf{B}$ is an isomorphism, as wanted.
\end{proof}

\subsubsection{Simply connected, solvable Lie
  groups} \label{sect:liealghull} Likewise, every simply connected,
solvable Lie group $G$ admits an algebraic hull: there exist an
$\R$-defined linear algebraic group $\mathbf{A}$ and an injective Lie
homomorphism $\iota \colon G \hookrightarrow \mathbf{A}$ such that
$\iota(G) \subseteq \mathbf{A}_\R$ and
\begin{itemize}
\item[(i)'] $\iota(G)$ is Zariski-dense in $\mathbf{A}$,
\item[(ii)'] $\mathbf{A}$ has a strong unipotent radical, i.e.,
  $\Cen_{\mathbf{A}}(\Rad_\textup{u}(\mathbf{A})) \subseteq
  \Rad_\textup{u}(\mathbf{A})$,
\item[(iii)'] $\dim \Rad_\textup{u}(\mathbf{A}) = \dim G$.
\end{itemize}
Again, the construction $\iota \colon G \hookrightarrow \mathbf{A}$ is
uniquely determined up to $\R$-isomorphism of linear algebraic groups.
We thus refer to $\mathbf{A}$, together with a possibly implicit
embedding $\iota$, as the \emph{algebraic hull} of~$G$.  It will be
convenient to denote the algebraic hull of $G$ by $\mathbf{A}_G$ and
to assume that $\iota$ is simply the inclusion map: $G \subseteq
\mathbf{A}_G$.

We observe that $\mathbf{A}_G$ is a connected solvable group.  A result
similar to Lemma~\ref{lem:extension} holds for the algebraic hull
$\mathbf{A}_G$ of~$G$.  The real algebraic group $A = \mathbf{A}_{\R}$,
together with the embedding $\iota \colon G \hookrightarrow A$, is
called the \emph{real algebraic hull} of $G$, and we denote it by~$A_G$.

The following result (see \cite[Proposition~2.3]{Ba04}) states, in
particular, that $G$ is a closed Lie subgroup in its real algebraic
hull~$A_G$.

\begin{pro}\label{pro:basics_real_hull}
  Let $G$ be a simply connected, solvable Lie group, and let $A = A_G$
  be its real algebraic hull.  Let $U$ denote the maximal unipotent
  normal subgroup of $A$, and let $T$ be a maximal reductive subgroup
  of~$A$.  Then $A = U \rtimes T$, and we denote by $\Upsilon \colon A
  \rightarrow U$, $g = ut \mapsto u$ the algebraic projection
  associated to the choice of~$T$.  Then $G$ is a closed normal
  subgroup of $A$ and $A = G \rtimes T$.  Moreover, $\Upsilon$
  restricts to a diffeomorphism $\Upsilon \vert_G \colon G \rightarrow
  U$.
\end{pro}

%%%

\subsection{Density properties of lattices }
Let $\Gamma$ be a lattice in a simply connected, solvable Lie
group~$G$. In this section we consider the Zariski-closure of $\Gamma$
in the algebraic hull of~$G$. Subsequently, we present several
applications to Zariski-dense lattices, which are more closely linked
to their ambient Lie groups than general lattices.

\subsubsection{General lattices}
We start with a simple criterion for recognising lattices in simply
connected, solvable Lie groups.

\begin{lem} \label{lem:criterion_lattice} Let $\Gamma$ be a discrete
  subgroup of a simply connected, solvable Lie group~$G$.  Then
  $\Gamma$ is poly-(infinite cyclic) and $\rank \Gamma \leq \dim G$.
  Moreover, $\Gamma$ is a lattice in $G$ if and only if $\rank \Gamma
  = \dim G$.
\end{lem}

\begin{proof}
  The group $\Gamma \cap \Nil(G)$ is discrete in $\Nil(G)$ and
  therefore a finitely generated, torsion-free, nilpotent
  group. Moreover, $\Gamma/(\Gamma \cap \Nil(G))$ embeds into the
  vector group $G/\Nil(G)$ and is a finitely generated, torsion-free,
  abelian group.  Thus $\Gamma$ is poly-(infinite cyclic).  The proof
  of \cite[Proposition~3.7]{Ra72} shows that $\rank \Gamma \leq \dim
  G$ and that equality holds if $\Gamma$ is a lattice in~$G$.
  Conversely, suppose that $\rank \Gamma = \dim G$.  We need to show
  that $\Gamma$ is cocompact in~$G$.  For this it is enough to prove
  that the cohomological dimension $\cdim \Gamma$ of $\Gamma$ is equal
  to $\dim G$; see~\cite[\S{}VIII.9.4]{Br94}.  This is true, because
  $\cdim \Gamma = \rank \Gamma$ by \cite[\S8.8]{Gr70}.
\end{proof}

Now let $\Gamma$ be a lattice in a simply connected, solvable Lie
group~$G$.  Then the inclusion $G \subseteq \mathbf{A}_{G}$ of $G$
into its algebraic hull $\mathbf{A}_{G}$ restricts to an inclusion
$\Gamma \subseteq \mathbf{A}_{G}$.

\begin{lem} \label{lem:udense} Let $\Gamma$ be a lattice in a simply
  connected, solvable Lie group~$G$.  Then $\Gamma$ is Zariski-dense
  in the unipotent part of $\mathbf{A}_{G}$:
  $$
  \up(\ac{\Gamma}) = \up( \mathbf{A}_{G}) =
  \Rad_\textup{u}(\mathbf{A}_G),
  $$
  where $\ac{\Gamma}$ denotes the Zariski-closure of $\Gamma$
  in~$\mathbf{A}_G$.
\end{lem}

\begin{proof}
  Since $\mathbf{A}_{G}$ is solvable, $\up( \mathbf{A}_{G}) =
  \Rad_\textup{u}(\mathbf{A}_G)$.  Clearly, $\up(\ac{\Gamma})$ is
  contained in~$\up( \mathbf{A}_{G})$. One of the properties of the
  algebraic hull is that $\dim \up( \mathbf{A}_{G}) = \dim G$.
  Therefore, to show that $\up(\ac{\Gamma}) = \up( \mathbf{A}_{G})$ it
  suffices to show that $\dim \up(\ac{\Gamma}) = \dim G$.

  Writing $A = \mathbf{A}_\R$ for the real algebraic hull of $G$, we
  have $A = U \rtimes T$, where $U = \up(A)$ is the maximal unipotent
  normal subgroup of $A$ and $T$ is a maximal reductive subgroup. The
  semidirect product decomposition $A = U \rtimes T$ induces a natural
  action $\alpha$ of $A$ on $U$:
  \begin{equation} \label{equ:action} \alpha(g) \cdot v = u \, ^{t\!}v
    = u t v t^{-1} \quad \text{for $g = ut \in U \rtimes T$ and $v \in
      U$.}
  \end{equation}
  By Proposition~\ref{pro:basics_real_hull}, we have $A = G \rtimes T$
  so that the action $\alpha$ in \eqref{equ:action} restricts to a
  simply transitive action of $G$ on~$U$.  Since $\Gamma \subseteq
  \ac{\Gamma}_{\R}$, the $\Gamma$-orbit $\alpha(\Gamma) \cdot 1$ is
  contained in $\up(\ac{\Gamma})_{\R}$, and, in fact, the group
  $\up(\ac{\Gamma})_{\R}$ is invariant under the action
  of~$\alpha(\Gamma)$.  Since $\Gamma$ acts properly discontinuously
  and freely on the contractible Lie group $\up(\ac{\Gamma})_{\R}$,
  the quotient $\Gamma \backslash \up(\ac{\Gamma})_{\R}$ is a compact
  manifold and an Eilenberg--MacLane space of type~$K(\Gamma,1)$.
  Since also $G/\Gamma$ is a compact Eilenberg--MacLane space of type
  $K(\Gamma,1)$, we deduce that $\dim \up(\ac{\Gamma}) = \dim G$.
\end{proof}

We remark that, with some extra work, the conclusion in
Lemma~\ref{lem:udense} can be strengthened: $\Gamma$ is Zariski-dense
in the real split part of~$\mathbf{A}_G$.  Under weaker assumptions,
we can formulate the following result.

\begin{lem} \label{lem:urank} Let $\Gamma$ be a discrete subgroup of a
  simply connected, solvable Lie group~$G$.  Then $\dim
  \up(\ac{\Gamma}) = \rank \Gamma$, where $\ac{\Gamma}$ denotes the
  Zariski-closure of $\Gamma$ in~$\mathbf{A}_G$.
\end{lem}

\begin{proof} The group $\Gamma$ is polycyclic.  Hence a lemma of
  Mostow (see \cite[Lemma~4.36]{Ra72}) shows that $\dim
  \up(\ac{\rho(\Gamma)})_{\R} \leq \rank \Gamma$ for any linear
  representation $\rho \colon \Gamma \rightarrow \GL_n(\R)$.  Applying
  this to the inclusion $\Gamma \subseteq A_G$ of $\Gamma$ into the
  real algebraic hull $A_G$ of $G$, we deduce that $\dim
  \up(\ac{\Gamma})_{\R} \leq \rank \Gamma$.  The (last part of the)
  proof of Lemma~\ref{lem:udense} shows that $\Gamma$ acts properly
  discontinuously and freely on the contractible Lie
  group~$\up(\ac{\Gamma})_{\R}$.  Therefore we have $\dim
  \up(\ac{\Gamma})_{\R} \geq \cdim \Gamma $, where $\cdim \Gamma$
  denotes the cohomological dimension of~$\Gamma$.  By
  \cite[\S8.8]{Gr70} we have $\cdim \Gamma = \rank \Gamma$.  Thus
  $\dim \up(\ac{\Gamma}) = \dim \up(\ac{\Gamma})_{\R} = \rank \Gamma$.
\end{proof}

\begin{pro} \label{pro:hull_incl} Let $\Gamma$ be a lattice in a
  simply connected, solvable Lie group~$G$.  Let $\mathbf{A}_\Gamma$
  be an $\R$-defined algebraic hull of $\Gamma$, and let
  $\mathbf{A}_G$ be an algebraic hull of~$G$.  Then the inclusion
  $\phi \colon \Gamma \hookrightarrow G \subseteq \mathbf{A}_G$
  extends uniquely to an $\R$-defined embedding of algebraic groups
  $\Phi \colon \mathbf{A}_{\Gamma} \hookrightarrow \mathbf{A}_{G}$,
  which restricts to an isomorphism of unipotent radicals.
\end{pro}

\begin{proof} Considering $\Gamma$ as a subgroup of $\mathbf{A}_G$,
  put $\mathbf{H} = \ac{\Gamma}$, the algebraic closure of $\Gamma$
  in~$\mathbf{A}_G$.  Since $\Gamma \subseteq (\mathbf{A}_G)_\R$, the
  group $\mathbf{H}$ is defined over~$\R$.  By Lemma~\ref{lem:udense},
  the group $\mathbf{H}$ has unipotent radical
  $\Rad_\textup{u}(\mathbf{H}) = \Rad_\textup{u}(\mathbf{A}_{G})$.  In
  particular, $\mathbf{H}$ has a strong unipotent radical.  Therefore,
  by Lemma~\ref{lem:extension}, the inclusion $\Gamma \subseteq
  \mathbf{H}$ extends to a surjective $\R$-defined morphism $\Psi
  \colon \mathbf{A}_{\Gamma} \rightarrow \mathbf{H}$ of algebraic
  groups.  Observe that $\dim \Rad_\textup{u}( \mathbf{A}_{\Gamma}) =
  \rank \Gamma = \dim G = \dim \Rad_\textup{u}(\mathbf{H})$.  Hence
  composing $\Psi$ with the inclusion $\mathbf{H} \subseteq
  \mathbf{A}_G$ yields a morphism $\Phi \colon \mathbf{A}_{\Gamma}
  \rightarrow \mathbf{A}_{G}$ which restricts to an isomorphism
  $\Rad_\textup{u}( \mathbf{A}_{\Gamma}) \rightarrow \Rad_\textup{u}(
  \mathbf{H}) = \Rad_\textup{u}( \mathbf{A}_G)$ of the unipotent
  radicals.  The kernel $\mathbf{K}$ of $\Phi$ intersects
  $\Rad_\textup{u}(\mathbf{A}_{\Gamma})$ trivially and hence
  centralises~$\Rad_\textup{u}(\mathbf{A}_{\Gamma})$.  Since
  $\mathbf{A}_{\Gamma}$ has a strong unipotent radical, we conclude
  that $\mathbf{K}$ is trivial.
\end{proof}

From $\rank \Gamma = \dim G$ we deduce that $\mathbf{A}_G$ is an
algebraic hull for $\Gamma$ if and only if $\Gamma$ is Zariski-dense
in~$\mathbf{A}_G$.  Of course, in general this need not be the case.
For instance, Example~\ref{exm:E2+} illustrates that the non-abelian
Lie group $\widetilde{E(2)^+}$ admits a host of abelian lattices.

\subsubsection{Zariski-dense lattices} We continue our discussion with
several applications to Zariski-dense lattices.

\begin{lem} \label{lem:commutators} Let $\Gamma$ be a Zariski-dense
  lattice in a simply connected, solvable Lie group~$G$. Then
  \begin{enumerate}
  \item $[\Gamma,\Gamma]$ is a lattice in the Lie group $[G,G]$;
  \item $\Gamma \cap [G,G]$ is a lattice in the Lie group $[G,G]$;
  \item The subgroup $\Gamma [G,G]$ is closed in~$G$.
  \end{enumerate}
\end{lem}

\begin{proof} 
  Let $U$ be the maximal unipotent subgroup of~$A_\Gamma = A_G$.
  Since $[G,G] \leq U$, it is a Zariski-closed subgroup of the real
  algebraic group~$A_\Gamma$.  Therefore, $[G,G] = [A_\Gamma,
  A_\Gamma]$, and each of the discrete subgroups $[\Gamma, \Gamma]$
  and $\Gamma \cap [G,G]$ is Zariski-dense in the latter.
  Consequently, $[\Gamma,\Gamma]$ and $\Gamma \cap [G,G]$ are
  (cocompact) lattices in~$[G,G]$.  It particular, it follows that the
  image of $[G,G]$ in the Hausdorff space $G/\Gamma$ is closed.
  Therefore, $\Gamma [G,G]$ is closed in~$G$.
\end{proof}

Next we characterise Zariski-dense lattices in terms of algebraic
hulls.
  
\begin{pro} \label{pro:Zdense} Let $\Gamma$ be a lattice in a simply
  connected, solvable Lie group~$G$.  Then $\Gamma$ is Zariski-dense
  in $G$ if and only if\/ $\Gamma$ is Zariski-dense subgroup
  of~$\mathbf{A}_G$.
\end{pro}

\begin{proof}
  Write $\mathbf{A} = \mathbf{A}_G$.  Since $G$ is simply connected,
  we have $\Ad(G) \cong G/\Zen(G)$.  Since we are working in
  characteristic $0$, we have $\Ad(\mathbf{A}) \cong
  \mathbf{A}/\Zen(\mathbf{A})$.  By the algebraic hull construction,
  $G$ is Zariski-dense in~$\mathbf{A}$.  Thus $\Zen(G) = G \cap
  \Zen(\mathbf{A})$ and $\Ad(G)$ embeds as a Zariski-dense subgroup $G
  \Zen(\mathbf{A}) / \Zen(\mathbf{A})$ into
  $\mathbf{A}/\Zen(\mathbf{A})$.

  Thus $\Gamma$ is Zariski-dense in $G$ if and only if $\ac{\Gamma}
  \Zen(\mathbf{A}) = \mathbf{A}$.  Hence to prove our claim it
  suffices to show that $\Zen(\mathbf{A}) \subseteq \ac{\Gamma}$.
  Since $\mathbf{A}$ has a strong unipotent radical, we have
  $\Zen(\mathbf{A}) \subseteq \Rad_\textup{u}(\mathbf{A})$.  Thus by
  Lemma~\ref{lem:udense} we have $\Zen(\mathbf{A}) \subseteq
  \ac{\Gamma}$.
\end{proof}

The next result shows, in particular, that extensions of automorphisms
are unique.

\begin{pro} \label{pro:ext_unique} Let $\Gamma$ be a Zariski-dense
  lattice in a simply connected, solvable Lie group $G$, and let $\phi
  \colon \Gamma \rightarrow G$.  If there exists $\Phi \in \Aut(G)$
  which extends $\phi$, then $\Phi$ is unique.
\end{pro}

\begin{proof}
  Assume that $\Phi_1 , \Phi_2 \colon G \rightarrow G$ are Lie
  automorphisms of $G$ such that $\Phi_1 \vert_\Gamma = \phi = \Phi_2
  \vert_\Gamma$.  Then by the Lie group version of
  Lemma~\ref{lem:extension} the maps $\Phi_1, \Phi_2$ extend to
  algebraic automorphisms $\tilde \Phi_1, \tilde \Phi_2$ of
  $\mathbf{A}_{G}$ which coincide on the Zariski-dense
  subgroup~$\Gamma$.  It follows that $\Phi_1 = \Phi_2$.
\end{proof}

If $\Gamma$ is a Zariski-dense lattice in a simply connected, solvable
Lie group $G$, then $\Gamma$ is a torsion-free polycyclic group, and
the next result shows that $\mathbf{A}_G$ (respectively $A_G$) also
constitutes an algebraic hull $\mathbf{A}_\Gamma$ (respectively real
algebraic hull $A_\Gamma$) of~$\Gamma$.

\begin{cor} \label{cor:AGamisAG} Let $\Gamma$ be a Zariski-dense
  lattice in a simply connected, solvable Lie group~$G$.  Then the
  following hold.
  \begin{enumerate}
  \item The inclusion $\Gamma \subseteq G$ extends uniquely to an
    $\R$-defined isomorphism of algebraic hulls $\mathbf{A}_{\Gamma}
    \rightarrow \mathbf{A}_{G}$.
  \item The inclusion $\Gamma \subseteq \mathbf{A}_{\Gamma}$ extends
    uniquely to an inclusion homomorphism of Lie groups $G
    \hookrightarrow (\mathbf{A}_{\Gamma})_\R$ and the image
    of $G$ is closed in~$(\mathbf{A}_\Gamma)_\R$.
\end{enumerate}
\end{cor}

\begin{proof} (1) By Proposition~\ref{pro:hull_incl}, the inclusion
  $\Gamma \subseteq G \subseteq \mathbf{A}_G$ extends uniquely to an
  $\R$-defined embedding $\Phi \colon \mathbf{A}_{\Gamma}
  \hookrightarrow \mathbf{A}_{G}$ of algebraic groups.  Since $\Gamma$
  is Zariski-dense, $\Phi(\mathbf{A}_\Gamma) = \mathbf{A}_{G}$, and
  $\Phi$ is a bijection.  Note that $\mathbf{A}_{G}$ is an
  $\R$-defined algebraic hull for~$\Gamma$.  Thus
  Corollary~\ref{cor:hull_unique} implies that $\Phi^{-1}$ is an
  $\R$-defined morphism, showing that $\Phi$ is an $\R$-defined
  isomorphism of algebraic hulls $\mathbf{A}_{\Gamma} \rightarrow
  \mathbf{A}_{G}$.

  (2) This is a direct consequence of (1) and
  Proposition~\ref{pro:basics_real_hull}.
\end{proof}

Let $\Gamma$ be a lattice in a simply connected, solvable Lie
group~$G$.  Recall that $\mathcal{X}(\Gamma,G)$ denotes the space of
all lattice embeddings of $\Gamma$ into~$G$.  We now consider,
temporarily, the following subspace
$$
\mathcal{X}^\textup{Z}(\Gamma,G) = \{ \phi \colon \Gamma
\hookrightarrow G \mid \text{$\phi(\Gamma)$ is a Zariski-dense lattice
  in $G$} \}.
$$
The next corollary shows that, if $\Gamma$ can be embedded as a
Zariski-dense lattice into $G$ at all, then
$\mathcal{X}^\textup{Z}(\Gamma,G) = \mathcal{X}(\Gamma,G)$.

\begin{cor} \label{cor:all_Zdense} Let $\phi_1, \phi_2 \colon \Gamma
  \hookrightarrow G$ be embeddings of $\Gamma$ as a lattice into a
  simply connected, solvable Lie group~$G$.  If $\phi_{1}(\Gamma)$ is
  Zariski-dense in~$G$, then $\phi_{2}(\Gamma)$ is also Zariski-dense
  in~$G$.
\end{cor}

\begin{proof} Suppose that $\phi_{1}(\Gamma)$ is Zariski-dense.  Let
  $\Phi_1, \Phi_2 \colon \mathbf{A}_\Gamma \hookrightarrow
  \mathbf{A}_G$ denote the extensions of $\phi_1, \phi_2$ to the level
  of algebraic hulls, provided by Proposition~\ref{pro:hull_incl}.  By
  Corollary~\ref{cor:AGamisAG}, $\Phi_{1} \colon \mathbf{A}_\Gamma
  \rightarrow \mathbf{A}_G$ is an isomorphism.  In particular, $\dim
  \mathbf{A}_\Gamma = \dim \mathbf{A}_G$, and this shows that the
  embedding $\Phi_{2} \colon \mathbf{A}_\Gamma \rightarrow
  \mathbf{A}_G$ is surjective.  Since $\Gamma$ is Zariski-dense
  in~$\mathbf{A}_\Gamma$, the group $\phi_2(\Gamma) =
  \Phi_{2}(\Gamma)$ is Zariski-dense in~$\mathbf{A}_{G}$. The claim
  follows from Proposition \ref{pro:Zdense}.
\end{proof}

The following corollary will be used repeatedly.

\begin{cor} Let $\Gamma$ be a Zariski-dense lattice in a simply
  connected, solvable Lie group~$G$.  Suppose that $\phi_1,\phi_2 \in
  \mathcal{X}(\Gamma,G)$.  Then there exists $\Psi \in
  \Aut_\R(\mathbf{A}_{G})$ such that $\phi_{2} = \Psi \circ \phi_{1}$.
\end{cor}

\begin{proof} Let $\Phi_1, \Phi_2 \colon \mathbf{A}_\Gamma \rightarrow
  \mathbf{A}_G$ denote the extensions of $\phi_1, \phi_2$ to the level
  of algebraic hulls; see Corollary~\ref{cor:AGamisAG}.  Since
  $\Phi_1$ and $\Phi_2$ are $\R$-defined isomorphisms of linear
  algebraic groups, and we may put $\Psi = \Phi_{2} \circ
  \Phi_{1}^{-1}$.
\end{proof}

%%%

\subsection{Tight Lie subgroups} \label{sec:tight_subgps} Throughout
this section let $\mathbf{A}$ be an $\R$-defined connected, solvable
linear algebraic group.  We write $\mathbf{A} = \mathbf{U} \rtimes
\mathbf{T}$, where $\mathbf{U} = \Rad_\textup{u}(\mathbf{A})$ and
$\mathbf{T}$ is a maximal $\R$-defined torus.  We remark that, if
$\mathbf{A}$ has a strong unipotent radical, then $\mathbf{U}$ is the
maximal nilpotent normal subgroup of~$\mathbf{A}$.  Let $A =
\mathbf{A}_{\R}$ denote the associated real algebraic group, and write
$A = U \rtimes T$, where $U = \mathbf{U}_\R$ is the maximal unipotent
normal subgroup of $A$ and $T = \mathbf{T}_\R$ is a maximal reductive
subgroup.  We observe that $[A,A] \subseteq U$.

\begin{dfn} \label{def:tight_subgroup} A Lie subgroup $G$ of $A$ is
  said to be \emph{tight} in $A$ if it is a connected, closed subgroup
  of $A$ and satisfies the following conditions: (i) $G$ is normal in
  $A$, (ii) $\dim G = \dim U$ and (iii) $G$ is Zariski-dense in the
  unipotent part of~$\mathbf{A}$, i.e., $\up(\ac{G}) = \mathbf{U}$,
  where $\ac{G}$ denotes the Zariski-closure in~$\mathbf{A}$.
\end{dfn}

We show that the set of tight Lie subgroups of $A$ can be
parametrised by Lie homomorphisms $\sigma \colon U \rightarrow T$
which are constant on $T$-orbits in $U$, i.e., which satisfy $[U,T]
\subseteq \ker \sigma$. (A similar result appears in
\cite[Theorem~5.3]{GrSe94}.)  As $T$ is abelian, $[U,U] \subseteq \ker
\sigma$, and hence the last condition is equivalent to $[A,A]
\subseteq \ker \sigma$.  For any such $\sigma$, we define
\begin{equation}\label{equ:Gtau}
  G_\sigma = \{ u \, \sigma(u) \mid u \in U \} 
\end{equation} 
and we observe that $\ker \sigma \subseteq G_\sigma$.

\begin{pro} \label{pro:tight_Lie_subs} For every Lie homomorphism
  $\sigma \colon U \rightarrow T$ such that $[U,T] \subseteq \ker
  \sigma$ the group $G_\sigma $, as defined in \eqref{equ:Gtau}, is a
  simply connected, tight Lie subgroup of~$A$.  Moreover, every tight
  Lie subgroup $G$ of $A$ is of this form.
\end{pro}

\begin{proof}
  Let $\sigma \colon U \rightarrow T$ be a Lie homomorphism such that
  $[U,T] \subseteq \ker \sigma$, and put $G = G_\sigma$.  Then $[A,A]
  \subseteq \ker \sigma \subseteq G$, and hence $G$ is a normal
  subgroup of~$A$.  We observe that $U \rightarrow G$, $u \mapsto u
  \sigma(u)$ is a diffeomorphism.  This shows that $\dim G = \dim U$ and
  also that $G$ is simply connected, because being a unipotent real
  algebraic group, $U$ is simply connected.  It remains to check that
  $\up(\ac{G}) = \mathbf{U}$.  Since $[A,A] \subseteq \ker \sigma$, we
  have $[\mathbf{A},\mathbf{A}] = \ac{[A,A]} \subseteq \ac{G}$.
  Consequently, we may assume that $\mathbf{A} = \mathbf{U} \times
  \mathbf{T}$.  But then $u \sigma(u)$ is the Jordan decomposition for
  any element of $G$, parametrised by $u \in U$, and consequently
  $\mathbf{U} = \ac{U} \subseteq \ac{G}$.

  Conversely, suppose that $G$ is a tight Lie subgroup of~$A$.  For
  each $g \in G$, we write $g = u_g t_g$ with $u_g \in U$ and $t_g \in
  T$.  Let $V = \{ u_g \mid g \in G \}$.  Since $G$ is connected, so
  is its continuous image~$V$.  The fact that $G$ is normal in $A$
  implies that $V$ forms a subgroup of~$U$.  Being a connected Lie
  subgroup of a unipotent group, $V$ is Zariski-closed in~$U$.
  Consequently, we have $V = \up(\ac{G})_\R = U$.  This shows that
  every $u \in U$ is of the form $u_{g}$, for some~$g \in G$.

  Observe that $G$ acts on $U$ via $u_g t_g \cdot v = u_g \,
  ^{t_g\!}v$.  The $G$-orbit of $1$ in this action is equal to $V = U$
  and thus $G/(G \cap T) \cong U$.  Since $\dim G = \dim U$, the map $G
  \rightarrow U$, $g \mapsto u_g$ is a covering map.  But $U$ is
  simply connected, hence it is a diffeomorphism.  We can thus define
  a smooth map $\sigma \colon U \rightarrow T$ by defining $\sigma(u_{g})
  = t_{g}$.  Since $G$ is a normal subgroup of $A$, a short
  computation reveals that $\sigma$ is both constant on $T$-orbits in
  $U$ and a homomorphism.
\end{proof}

\begin{cor} \label{cor:tightissc}
  Every tight Lie subgroup of $A$ is simply connected.
\end{cor}

Next we prove a relative version of
Proposition~\ref{pro:tight_Lie_subs}.

\begin{pro} \label{pro:tight_Lie_subs_relative} Suppose that $H$ is a
  tight Lie subgroup of~$A$.  Then the collection of tight Lie
  subgroups of $A$ is in one-to-one correspondence with the space
  $\Hom_{\textup{Lie}}(H/[H,H],T)$ of Lie homomorphisms.  Moreover, if
  $G$ corresponds to $\overline{\sigma_{G \mid H}}$, where $\sigma_{G
    \mid H} \in \Hom_{\textup{Lie}}(H,T)$ with $[H,H] \subseteq \ker
  \sigma_{G \mid H}$, then $G \cap H = \ker \sigma_{G \mid H}$.
\end{pro}

\begin{proof}
  According to Proposition~\ref{pro:tight_Lie_subs}, every tight Lie
  subgroup $G$ of $A$ is simply connected and gives rise to a Lie
  homomorphism $\sigma_G \colon U \rightarrow T$ with $[A,A] \subseteq
  \ker \sigma_G$.  In particular, this is the case for~$H$.

  The required bijection is defined as follows.  Let $G$ be a tight
  Lie subgroup of~$A$.  We define a map $\sigma_{G \mid H} \colon H
  \rightarrow T$ as follows.  Let $h \in H$.  Then $h = u \sigma_H(u)$
  for a unique $u \in U$.  Set $\sigma_{G \mid H}(h) =
  \sigma_H(u)^{-1} \sigma_G(u)$.  A short computation, using that
  $\sigma_H$ and $\sigma_G$ are constant on $T$-orbits, shows that the
  map $\sigma_{G \mid H} \colon H \rightarrow T$ is a Lie
  homomorphism.  Since $T$ is abelian, we have $[H,H] \subseteq \ker
  \sigma_{G \mid H}$, and clearly $G \cap H = \ker \sigma_{G \mid H}$.

  Conversely, starting from $\sigma \in \Hom_{\textup{Lie}}(H,T)$ with
  $[H,H] \subseteq \ker \sigma$ we define $G = \{ h \sigma(h) \mid h
  \in H \}$.  A short computation shows that $G$ is a  tight
  Lie subgroup of $A$, and that the resulting map is the desired
  inverse.
\end{proof}

Recall that a simply connected, solvable Lie group $G$ is called
unipotently connected if its maximal nilpotent normal subgroup is
connected; see Section~\ref{sec:uconnected}.
 
\begin{lem} \label{lem:Gtau_properties} Let $\sigma \colon U
  \rightarrow T$ be a Lie homomorphism such that $[U,T] \subseteq \ker
  \sigma$.  Then $G = G_\sigma$, as defined in \eqref{equ:Gtau}, has
  the following properties:
  \begin{enumerate}
  \item $A = G \rtimes T$;
  \item $\ker \sigma = G \cap U $;
  \item $G$ is Zariski-dense in $\mathbf{A}$ if and only if the group
    $\sigma(U)$ is Zariski-dense in~$\mathbf{T}$.
  \end{enumerate}
  Assuming that $G = G_\sigma$ is Zariski-dense in $\mathbf{A}$ and that
  $\mathbf{A}$ has a strong unipotent radical, the group also
  satisfies:
  \begin{enumerate}
  \item[(4)] $\Nil(G) = (\ker \sigma)_\circ$;
  \item[(5)] $G$ is unipotently connected if and only if $\ker \sigma$ is
    connected.
  \end{enumerate}
\end{lem}

\begin{proof}
  By Proposition~\ref{pro:tight_Lie_subs}, $G$ is a simply connected,
  tight Lie subgroup of~$A$.  In particular, we have $G
  \trianglelefteq A$.  Clearly, $G \cap T = \{1\}$, $GT = UT = A$ and
  hence $A = G \rtimes T$.  This proves (1).

  Since $G$ is tight in $A$, we have $\mathbf{U} \subseteq \ac{G}$.
  Hence $G$ is Zariski-dense in $\mathbf{A}$ if and only if
  $G\mathbf{U}/\mathbf{U}$ is Zariski-dense in
  $\mathbf{A}/\mathbf{U}$.  The latter is equivalent to $\sigma(U)$
  being Zariski-dense in~$\mathbf{T}$.  This proves (3).

  Put $K = \ker \sigma$.  Clearly, $K = G \cap U$, justifying (2).
  Now suppose that $G$ is Zariski-dense in $\mathbf{A}$ and that
  $\mathbf{A}$ has a strong unipotent radical.  Then $\mathbf{U}$ is
  the maximal nilpotent normal subgroup of $\mathbf{A}$ and
  consequently $K = G \cap U$ is the maximal nilpotent normal subgroup
  of~$G$.  This implies $\Nil(G) = K_\circ$, viz.\ (4), and shows that
  $G$ is unipotently connected if and only if $K$ is connected,
  viz.~(5).
\end{proof}

\begin{lem}\label{lem:Gnormal}
  Suppose that $G$ is a connected Lie group which is contained and
  Zariski-dense in~$A$.  Then $[A,A]$ is contained in~$\Nil(G)$.  In
  particular, $G$ is normal in~$A$.
\end{lem}

\begin{proof} The commutator subgroup $[G,G]$ is connected and
  contained in~$\Nil(G)$. In particular, it is unipotent, and
  therefore Zariski-closed in~$A$.  Since $G$ is Zariski-dense in $A$,
  it follows that $[A,A] = [G,G] \subseteq \Nil(G)$.
\end{proof}
 
\begin{lem} \label{lem:desG} Let $G$ be a simply connected, solvable
  Lie group, and suppose that $A = A_G$ is the real algebraic hull
  of~$G$.  Then $G = G_\sigma$, as defined in \eqref{equ:Gtau}, for a
  suitable Lie homomorphism $\sigma \colon U \rightarrow T$ such that
  $[U,T] \subseteq \ker \sigma$.
\end{lem}

\begin{proof} By Proposition~\ref{pro:tight_Lie_subs}, it suffices to
  show that $G$ is tight in~$A$.  By Lemma~\ref{lem:Gnormal}, the
  group $G$ is normal in~$A$.  By the definition of the algebraic
  hull, we have $\dim G = \dim U$, and because $G$ is Zariski-dense in
  $\mathbf{A}$, we surely have $\up(\ac{G}) = \mathbf{U}$.
\end{proof}

%We conclude the section with the following result.\footnote{BK: This
%  may become redundant?  Check where it is referred to.}
%
%\begin{pro} \label{pro:AutG} Let $\sigma \colon U \rightarrow T$ be a
%  Lie homomorphism such that $[U,T] \subseteq \ker \sigma$.  Let $G =
%  G_\sigma \leq \mathbf{A}$, as defined in~\eqref{equ:Gtau}.  Let
%  $\mathbf{N} = \ac{\Nil(G)}$ in $\mathbf{A}$, and let $\Phi \in
%  \Aut_\R(\mathbf{A})$ such that $\Phi(\mathbf{N}) = \mathbf{N}$.  If
%  the induced automorphism $\Phi_{\mathbf{A}/\mathbf{N}}$ on the
%  quotient $\mathbf{A}/\mathbf{N}$ satisfies
%  $\Phi_{\mathbf{A}/\mathbf{N}} = \id_{\mathbf{A}/\mathbf{N}}$, then
%  $\Phi(G) = G$.
%\end{pro}
%
%\begin{proof}
%  It suffices to show that $\Phi(G) \subseteq G$.  Note that $\Nil(G)$
%  is a connected Lie subgroup of the unipotent real algebraic group
%  $U$, hence Zariski-closed in~$U$.  We write $N = \mathbf{N}_\R =
%  \Nil(G)$.  From $\Phi_ {\mathbf{A}/\mathbf{N}} =
%  \id_{\mathbf{A}/\mathbf{N}}$ we deduce that $\Phi(G) \subseteq G N$.
%  Since $N \subseteq \ker \sigma \subseteq G$ by
%  Lemma~\ref{lem:Gtau_properties}, we conclude that $\Phi(G) \subseteq
%  G$ as wanted.
%\end{proof}

%%%%%

\section{The Fitting subgroup and Mostow's
  theorem} \label{sect:Fitting}

Let $\Gamma$ be a lattice in a simply connected, solvable Lie
group~$G$.  Consider the induced inclusion of real algebraic hulls
$A_\Gamma \subseteq A_G$.  Let $U$ be the unipotent radical of
$A_\Gamma$, which by Proposition \ref{pro:hull_incl} is also the
unipotent radical of~$A_G$. Then the closure $F= \ac{\Fitt(\Gamma)}$
of the Fitting subgroup of $\Gamma$ in $A_\Gamma$ is a nilpotent
normal subgroup of~$A_\Gamma$.  Since $A_\Gamma$ has a strong
unipotent radical, we have $F \leq U$.  In particular, $\Fitt(\Gamma)$
is a subgroup of $G \cap U= \up(G)$. Therefore we have
$$
\Fitt(\Gamma) = \up(\Gamma).
$$
Let $N = \Nil(G)$ denote the nilradical of~$G$. Note that $N =
\up(G)_{\circ}$ is the identity component of~$\up(G)$.  Mostow's
theorem~\cite[\S5]{Mo57}, asserts that $\Gamma \cap N$ is a lattice
in~$N$.  Since $\Gamma \cap N$ is a nilpotent normal subgroup of
$\Gamma$, it is contained in~$\Fitt(\Gamma)$.  In the following we
strengthen Mostow's theorem: Proposition~\ref{pro:fittislattice} below
shows that $\Fitt(\Gamma)$ is a lattice in the closed
subgroup~$\up(G)$.

We fix some further notation.  Let $T$ be a maximal reductive subgroup
of~$A_G$.  Then $A_G = U \rtimes T$ is an $\R$-defined splitting of
real algebraic groups.  Let $\pi \colon A_G \rightarrow T$ be the
projection homomorphism associated to this splitting.  The real
algebraic torus $T$ admits a direct decomposition $T = T_\text{s}
\times T_\text{c}$, where $T_\text{s}$ is a subtorus split over the
reals and $T_\text{c}$ is the maximal compact subgroup of~$T$.  In
particular, $T_\text{s}$ is a simply connected group.  Based on the
decomposition of $T$, we write $\pi = \pi_\text{s} \times
\pi_\text{c}$ with homomorphisms $\pi_\text{s}\colon A_G \rightarrow
T_\text{s}$ and $\pi_\text{c}\colon A_G \rightarrow T_\text{c}$.  
%The following lemma should be compared
%with~\cite[Proposition~?]{Au73}.\footnote{BK: fill in reference}

\begin{lem} \label{lem:splitimage} With the notation introduced above,
  the following hold:
  \begin{enumerate}
  \item $\pi(\Gamma)$ is discrete in $T$, and $\pi_\text{s}(\Gamma)$ is
    discrete in $T_\text{s}$,
  \item $\pi(G)$ is a closed subgroup of $T$,
  \item $\pi(\Gamma)$ is a lattice in $\pi(G)$, and $\pi_\text{s}(\Gamma)$
    is a lattice in~$\pi_\text{s}(G)$.
  \end{enumerate}
\end{lem}

\begin{proof}
  The $\Q$-defined algebraic group $\mathbf{A}_\Gamma$ admits a
  semidirect decomposition $\mathbf{A}_\Gamma = \mathbf{U} \rtimes
  \mathbf{S}$ over $\Q$, where $\mathbf{U} =
  \Rad_\textup{u}(\mathbf{A}_\Gamma)$ denotes the unipotent radical
  and $\mathbf{S}$ is a $\Q$-defined maximal algebraic torus.  The
  projection $\rho \colon \mathbf{A}_\Gamma \rightarrow \mathbf{S}$
  associated to this decomposition is a $\mathbb{Q}$-defined
  homomorphism of algebraic groups.  Fix a $\Q$-defined embedding of
  $\mathbf{S}$ into a general linear group~$\GL_n$.  Then the
  arithmetic subgroup $\mathbf{S}(\Z) = \mathbf{S} \cap \GL_n(\Z)$ is
  a discrete subgroup of $S = \mathbf{S}_\R$.
  % see \cite[Lemma~4.1]{PlRa94}.  
  By the additional property (iv) of the algebraic hull
  $\mathbf{A}_{\Gamma}$, the group $\Gamma \cap
  \rho^{-1}(\mathbf{S}(\Z))$ is of finite index in~$\Gamma$.
  Consequently, $\rho(\Gamma)$ is discrete in~$S$.  Since
  $\mathbf{A}_\Gamma \leq \mathbf{A}_G$, the real algebraic torus $S$
  identifies with a subtorus of $T$ and $\rho(\Gamma)$
  with~$\pi(\Gamma)$.  Therefore, $\pi(\Gamma)$ is discrete in~$T$.
  Furthermore, since $T_\text{c}$ is compact, $\pi_\text{s}(\Gamma)$,
  i.e., the image of $\pi(\Gamma)$ in $T_\text{s}$, is discrete
  in~$T_\text{s}$.  Thus (1) holds.

  Being the continuous image of the compact space $G/\Gamma$ in the
  Hausdorff space $T/\pi(\Gamma)$, the group $\pi(G)/\pi(\Gamma)$ is
  closed in~$T/\pi(\Gamma)$.  Hence its pre-image $\pi(G)$ under the
  map $T \rightarrow T/\pi(\Gamma)$ is closed in $T$, and (2) holds.

  Since $\pi(G)/ \pi(\Gamma)$ is the continuous image of the compact
  space $G/\Gamma$, the group $\pi(\Gamma)$ is cocompact in~$\pi(G)$.
  Similarly, $\pi_\text{s}(\Gamma)$ is cocompact in~$\pi_\text{s}(G)$.  In
  view of (1), we deduce that (3) holds.
\end{proof}

We continue to use the notation set up before Lemma
\ref{lem:splitimage}.  By Lemma \ref{lem:desG}, there exists a Lie
homomorphism $\sigma \colon U \rightarrow T$ such that every $g \in G$
can be expressed as $g= u \sigma(u)$, in accordance with the
decomposition $A_G = U \rtimes T$. We observe that $\pi(g) =
\sigma(u)$ for all $g \in G$.  Similarly as for $\pi$, we may
decompose $\sigma = \sigma_\text{s} \times \sigma_\text{c}$ into Lie
homomorphisms $\sigma_\text{s}\colon U \rightarrow T_\text{s}$ and
$\sigma_\text{c}\colon U \rightarrow T_\text{c}$.

\begin{lem} \label{lem:u(G)_sigma} With the notation introduced above,
  we have 
  $$
  \ac{\up(G)} = \ker \sigma_\text{s},
  $$
  where $\ac{\up(G)}$ denotes the Zariski-closure of $\up(G)$
  in~$A_G$.  

  Furthermore, $\ac{\up(G)}$ is a simply connected, nilpotent Lie
  group.
\end{lem}

\begin{proof}
  Clearly, $\up(G) = U \cap G = \ker \sigma$.  Since
  $\sigma_\text{s}\colon U \rightarrow T_\text{s}$ is a Lie
  homomorphism of simply connected groups, $\ker \sigma_\text{s}$ is a
  connected subgroup of $U$, thus it is simply connected and
  Zariski-closed.  Observe further that
  $$ 
  \ker \sigma_\text{s}/ \up(G) \cong \sigma_\text{c}(\ker
  \sigma_\text{s}) = \pi(G) \cap T_\text{c} \; .
  $$ 
  By (2) of Lemma~\ref{lem:splitimage}, this implies that $ \ker
  \sigma_\text{s}/ \up(G)$ is compact.  Being a cocompact subgroup of the
  unipotent real algebraic group $\ker \sigma_\text{s}$, the group $\up(G)$
  is Zariski-dense in~$\ker \sigma_\text{s}$.
\end{proof}

\begin{pro} \label{pro:fittislattice} Let $\Gamma$ be a lattice in a
  simply connected, solvable Lie group~$G$.  Then, taking
  Zariski-closures in $A_\Gamma$, the following hold:
  \begin{enumerate}
  \item $\Fitt(\Gamma) =\Gamma \cap \ac{\up(G)}$ is a lattice in
    $\ac{\up(G)}$,
  \item $\ac{\Fitt(\Gamma)} = \ac{\up(G)}$,
  \item $\Fitt(\Gamma)$ is a lattice in~${\up(G)}$.
  \end{enumerate}
\end{pro}

\begin{proof}
  By (3) of Lemma~\ref{lem:splitimage}, the group $\Gamma (\ker
  \pi_\text{s} \cap G)$ is closed in~$G$.  Hence, by
  \cite[Theorem~1.13]{Ra72}, the group $\ker \pi_\text{s} \cap \Gamma$
  is a lattice in the simply connected group $\ker \pi_\text{s} \cap
  G$.  For $g=u \sigma(u) \in G$ we have $g \in \ker \pi_\text{s}$ if
  and only if $u \in \ker \sigma_\text{s}$.  Thus we deduce from
  Lemma~\ref{lem:criterion_lattice} and Lemma~\ref{lem:u(G)_sigma}
  that
  $$
  \rank (\ker \pi_\text{s} \cap \Gamma) = \dim (\ker \pi_\text{s} \cap
  G) = \dim (\ker (\sigma_\text{s})) = \dim (\ac{\up(G)}).
  $$
  Observe further that $\pi(\ker \pi_\text{s} \cap {\Gamma}) = \pi(\Gamma)
  \cap T_\text{c}$ is finite, since $\pi(\Gamma)$ is discrete in $T$, by
  Lemma \ref{lem:splitimage}.  Therefore, $\Fitt(\Gamma) = \up(\Gamma)
  = \ker \pi \cap \Gamma$ is of finite index in $\ker \pi_\text{s} \cap
  {\Gamma}$. This implies that 
  $$
  \rank \Fitt(\Gamma) = \rank (\ker \pi_\text{s} \cap {\Gamma}) = \dim
  (\ac{\up(G)}).
  $$
  As $\Fitt(\Gamma) = \up(\Gamma) \subseteq \ac{\up(G)}$, the Fitting
  subgroup $\Fitt(\Gamma)$ is a discrete subgroup of the simply
  connected, nilpotent group $\ac{\up(G)}$; see
  Lemma~\ref{lem:u(G)_sigma}.  Hence Lemma~\ref{lem:criterion_lattice}
  shows that $\Fitt(\Gamma)$ is a lattice in $\ac{\up(G)}$, giving (1)
  and~(2).  The proof of Lemma~\ref{lem:u(G)_sigma} showed that
  $\up(G)$ is cocompact in $\ac{\up(G)}$, implying~(3).
\end{proof} 

%%%%%

\section{Unipotently connected groups}
\label{sec:uconnected}

Let $G$ be a simply connected, solvable Lie group with algebraic
hull~$\mathbf{A}_G$.  As $\mathbf{A}_G$ is solvable, we have
$\Rad_\textup{u}(\mathbf{A}_G) = \up(\mathbf{A}_G)$.

\begin{dfn}
  The Lie group $G$ is called \emph{unipotently connected}, if the
  closed subgroup $\up(G) = G \cap \Rad_\textup{u}(\mathbf{A}_G)$ of
  $G$ is connected.
\end{dfn}

Since $ \mathbf{A}_G$ has a strong unipotent radical, $\up(G)$ is in
fact the maximal nilpotent normal subgroup (the discrete nilradical)
of~$G$.  Therefore, equivalently, $G$ is unipotently connected if and
only if $\up(G) = \Nil(G)$.  This shows that the class of unipotently
connected groups coincides with the class of groups (A) introduced by
Starkov; see~\cite[Proposition~1.12]{St94}.  Groups of type (A) are
defined in terms the eigenvalues of the adjoint representation.

\smallskip

The following lemma shows that the group theoretic structure of
lattices in unipotently connected groups is slightly restricted.

\begin{lem} \label{lem:fitt_quo_tor_free}
  Let $G$ be a simply connected, solvable Lie group which is
  unipotently connected.  Then for every lattice $\Gamma$ in $G$ the
  Fitting quotient $\Gamma/\Fitt(\Gamma)$ is torsion-free.
\end{lem}

\begin{proof} 
  In fact, $\Fitt(\Gamma) = \Gamma \cap \up(G)$, see the beginning of
  Section~\ref{sect:Fitting}.  Since $G$ is unipotently connected,
  $\Fitt(\Gamma) = \Gamma \cap \Nil(G)$, and therefore
  $\Gamma/\Fitt(\Gamma)$ embeds into the abelian vector group $V=
  G/\Nil(G)$.
\end{proof}

This provides a genuine restriction.  For instance, using the notation
of Example~\ref{exm:E2+}, the lattice $\Gamma = (\Z + \Z
i).X(\Z[\tfrac{1}{2}])$ of the simply connected group $G = V . X(\R)
\cong \widetilde{E(2)^+}$ has Fitting quotient $\Gamma/\Fitt(\Gamma)
\cong \Z / 2\Z$.  We also remark that as a side product of the
construction given in Example~\ref{exm:nicht_u-zushg} we obtain the
following.  In the notation used there, the Zariski-dense lattice
$\Delta = \eta(\Lambda) X(\Z) Y(2\Z)$ of the Lie group $G$ is not
unipotently connected, despite $\Delta/\Fitt(\Delta)$ being
torsion-free.  Such examples can also be constructed in an easier
fashion, starting from the example of Auslander used in
Example~\ref{exm:factorial}.

\begin{pro} \label{pro:exists_unipotently_connected} Let $\Gamma$ be a
  lattice in a simply connected, solvable Lie group~$G$.  Then
  $\Gamma$ has a finite index subgroup which embeds as a Zariski-dense
  lattice into a simply connected, solvable Lie group $H$ which, in
  addition, is unipotently connected.
\end{pro}

\begin{proof}
  In fact, $\Gamma$ has a finite index subgroup $\Delta$ such that
  $\Delta / \Fitt(\Delta)$ is torsion-free and $\Delta \leq
  (A_\Delta)_{\circ}$ is contained in the identity component of its
  real algebraic hull.  The claim follows by a simple construction
  which is exhibited in \cite[Chap.~III \S 7, p.~250--251]{Au73}.  See
  also \cite[Proposition~4.1]{GrSe94}, or \cite[Chap.~1,
  Proposition~1.15]{Ba08} for special cases.
\end{proof}

From Proposition~\ref{pro:fittislattice} we derive the following
corollary.

\begin{cor} \label{cor:nil_and_fitt} Let $G$ be a simply connected,
  solvable Lie group, and let $\Gamma$ be a lattice in~$G$. Then the
  following conditions are equivalent:
  \begin{enumerate}
  \item $G$ is unipotently connected,
  \item $\rank \Fitt(\Gamma) = \dim \Nil(G)$,
  \item $\ac{\Fitt(\Gamma)} = \Nil(G)$, where the Zariski-closure is
    taken in $A_\Gamma$,
  \item $\Fitt(\Gamma)$ is a lattice in~$\Nil(G)$.
  \end{enumerate}
\end{cor}

\begin{proof}
  By definition, the group $G$ is unipotently connected if and only if
  $\up(G) = \Nil(G) (= \up(G)_{\circ})$. Observe that $\up(G)/
  \up(G)_{\circ}$ is a discrete subgroup of the vector group
  $G/\Nil(G)$, since $\up(G)$ is a closed subgroup of~$G$. Therefore
  $\up(G)/ \up(G)_{\circ}$ is finitely generated (and abelian) of rank
  $\dim \ac{\up(G)} - \dim \up(G)_{\circ}$.  By
  Proposition~\ref{pro:fittislattice}, the group $\Fitt(\Gamma)$ is a
  lattice in $\up(G)$, which implies that
  \begin{align*}
    \rank \Fitt(\Gamma)& = \dim \up(G)_{\circ} + \rank (\up(G)/
    \up(G)_{\circ}) \\
    & = \dim \Nil(G) + \rank (\up(G)/ \up(G)_{\circ}).
  \end{align*}
  Since $\rank (\up(G)/ \up(G)_{\circ}) = 0$ if and only if $\up(G) =
  \up(G)_{\circ}$, this shows that (1) and (2) are equivalent.
  
  Writing $F = \ac{\Fitt(\Gamma)}$, part (2) of
  Proposition~\ref{pro:fittislattice} states that $F = \ac{\up(G)}$.
  Therefore, if $G$ is unipotently connected, $F = \up(G) = \Nil(G)$.
  Hence, (1) implies~(3).  Proposition~\ref{pro:fittislattice} shows
  that (3) implies~(4).  Finally, Lemma~\ref{lem:criterion_lattice}
  yields that (4) implies~(2).
  \end{proof}

%%%%%

\section{Description of the space $\mathcal{G}(\Gamma)$}
\label{sect:GGamma} Throughout this section, let $\Gamma$ be a
torsion-free polycyclic group and suppose that its algebraic hull
$\mathbf{A} = \mathbf{A}_\Gamma$ is connected.  Write $\mathbf{A} =
\mathbf{U} \rtimes \mathbf{T}$, where $\mathbf{U} =
\Rad_\textup{u}(\mathbf{A})$ and $\mathbf{T}$ is a maximal
$\R$-defined torus.  Put $A = \mathbf{A}_\R$, $U = \mathbf{U}_\R$ and
$T = \mathbf{T}_\R$.  Then $A = U \rtimes T$.

Recalling from Section~\ref{sec:tight_subgps}, in particular
Definition~\ref{def:tight_subgroup}, the notion of a tight subgroup,
we investigate the set
\begin{equation} \label{eq:GGam} 
  \mathcal{G}(\Gamma) = \{ G \mid
  \text{$G$ a tight Lie subgroup of $A$ with $\Gamma \subseteq G$} \}.
\end{equation}
Our interest in $\mathcal{G}(\Gamma)$ is founded on
Proposition~\ref{pro:real_structureset}, which will show that
$\mathcal{G}(\Gamma)$ captures all possible embeddings of $\Gamma$ as
a Zariski-dense lattice into simply connected, solvable Lie groups.
The following discussion, in particular \eqref{equ:GGam_descrip},
already provides an indication of the relevance
of~$\mathcal{G}(\Gamma)$.

We remark that any connected Lie subgroup $G$ of $A$ containing
$\Gamma$ is Zariski-dense in $A$, because $\Gamma$ is Zariski-dense in
$A$, and it is normal in $A$ by Lemma~\ref{lem:Gnormal}. Since $\dim U
= \rank \Gamma$, this shows that
\begin{multline*}
  \mathcal{G}(\Gamma) = \{ G \mid \text{$G$ a connected, closed Lie
    subgroup of $A$} \\
  \text{such that $\Gamma \subseteq G$ and $\dim G = \rank \Gamma$}
  \}.
\end{multline*}
Throughout this section we suppose that $\mathcal{G}(\Gamma)$ is not
empty.

\begin{lem} \label{lem:Gamma_is_lattice} Let $G \in
  \mathcal{G}(\Gamma)$.  Then $\mathbf{A}$ is an algebraic hull of
  $G$, and $\Gamma$ is a Zariski-dense lattice in~$G$.
\end{lem}

\begin{proof}
  Clearly, $\Gamma$ is a discrete subgroup of~$G$.  From
  Lemma~\ref{lem:criterion_lattice} and $\rank \Gamma = \dim G$ we
  deduce that $\Gamma$ is a lattice in~$G$.  Since $\Gamma$ is
  Zariski-dense in $\mathbf{A}$, so is~$G$.  Hence $\mathbf{A} =
  \mathbf{A}_G$ and $\Gamma$ is Zariski-dense in $G$ by
  Proposition~\ref{pro:Zdense}.
\end{proof}

Lemma~\ref{lem:Gamma_is_lattice} and Corollary~\ref{cor:tightissc}
show that
\begin{multline}
\label{equ:GGam_descrip}
\mathcal{G}(\Gamma) = \{ G \mid \text{$G$ a simply connected,
  solvable Lie subgroup} \\
  \text{of $A$ such that $\Gamma$ is a Zariski-dense lattice in $G$}
  \}.
\end{multline}

\begin{pro} \label{pro:description_G_Gamma}
  Let $H \in \mathcal{G}(\Gamma)$.  Then there is a bijection
  $$
  \mathcal{G}(\Gamma) \rightarrow
  \Hom_{\textup{Lie}}(H/\Gamma[H,H],T).
  $$
  In particular, $\mathcal{G}(\Gamma)$ is either equal to $\{ H \}$ or
  it is countably infinite.
\end{pro}

\begin{proof}
  Clearly, we may assume that $H$ is not trivial.
  Proposition~\ref{pro:tight_Lie_subs_relative} implies that there is
  a bijection from $\mathcal{G}(\Gamma)$ onto
  $\Hom_{\textup{Lie}}(H/\Gamma[H,H],T)$, because $\Gamma \subseteq G
  \cap H \subseteq \ker \sigma_{G \mid H}$ for every $G \in
  \mathcal{G}(\Gamma)$.  Now $H/\Gamma[H,H]$ is a compact torus. In
  fact, $H/\Gamma[H,H]$ is Hausdorff, since Zariski-denseness of
  $\Gamma$ implies that $\Gamma [H,H]$ is closed in $H$; see
  Lemma~\ref{lem:commutators}. Since $H/\Gamma$ is compact, so
  is~$H/\Gamma[H,H]$.  Moreover, the latter quotient has dimension
  $d$, where $d = \dim H - \dim [H,H] > 0$. Let $T_\text{c}$ be the
  maximal compact subtorus of $T$, and $t= \dim T_\text{c}$. Then,
  since the group $\Hom_{\textup{Lie}}(H/\Gamma[H,H],T)$ is isomorphic
  to $\Hom(\Z^d,\Z^t)$ it is either a point ($t=0$) or countably
  infinite.
\end{proof}

\begin{cor}\label{cor:charac_real_type}
  Let $H \in \mathcal{G}(\Gamma)$.  Then the Lie group $H$ is of real
  type if and only if $\mathcal{G}(\Gamma) = \{ H \}$.
\end{cor}

\begin{proof}
  According to Proposition~\ref{pro:description_G_Gamma}, we have
  $\mathcal{G}(\Gamma) = \{ H \}$ if and only if the compact part
  $T_\text{c}$ of the connected torus $T$ is trivial.  This is the
  case if and only if the algebraic group $\mathbf{T}$ is $\R$-split.
  Observe that $H$ is of real type if and only if there is a maximal
  torus in the Zariski-closure of its adjoint image which is
  $\R$-split.  Since this Zariski-closure coincides with the adjoint
  image of $\mathbf{A}$ (compare also the proof of
  Proposition~\ref{pro:Zdense}), this is the case if and only if
  $\mathbf{T}$ is $\R$-split.
\end{proof}

We fix some further notation.  Let $\mathbf{F} = \ac{\Fitt(\Gamma)}$
denote the Zariski-closure of $\Fitt(\Gamma)$ in $\mathbf{A} =
\mathbf{A}_\Gamma$, and let $F = \mathbf{F}_\R = A \cap \mathbf{F}$.
Since $\mathbf{F}$ is a nilpotent normal subgroup of $\mathbf{A}$ and
since $\mathbf{A}$ has a strong unipotent radical, we have $\mathbf{F}
\subseteq \mathbf{U}$ and accordingly $F \subseteq U$.

\begin{pro} \label{pro:nil=nil}
  Let $G,H \in \mathcal{G}(\Gamma)$, and suppose that $G$ is
  unipotently connected.  Then $\Nil(H) \subseteq \Nil(G)$.
\end{pro}

\begin{proof}
  From Proposition~\ref{pro:fittislattice} and
  Corollary~\ref{cor:nil_and_fitt} we deduce that $\Nil(H) \subseteq
  \up(H) \subseteq F = \Nil(G)$.
\end{proof}

Next we adapt the description in
Proposition~\ref{pro:description_G_Gamma} to unipotently connected
groups.  For this purpose we consider the subset
$$
\mathcal{G}^\text{uc}(\Gamma) = \{ G \in \mathcal{G}(\Gamma) \mid
\text{$G$ unipotently connected} \}
$$
of all unipotently connected groups in~$\mathcal{G}(\Gamma)$. 

We assume in the following that $\mathcal{G}^\textup{uc}(\Gamma) \not =
\varnothing$.  This condition puts further restrictions on $\Gamma$,
but can always be met by passing to a finite index subgroup; see
Proposition~\ref{pro:exists_unipotently_connected}.

\begin{pro} \label{pro:description_u-connected_G_Gamma} Let $H \in
  \mathcal{G}^\textup{uc}(\Gamma)$, and set $N = \Nil(H)$.  Then the
  bijection $\mathcal{G}(\Gamma) \rightarrow
  \Hom_{\textup{Lie}}(H/\Gamma [H,H],T)$ established in
  Proposition~\ref{pro:description_G_Gamma} induces a bijection
  $$
  \mathcal{G}^\textup{uc}(\Gamma) \rightarrow
  \Hom_{\textup{Lie}}(H/\Gamma N,T).
  $$
  In particular, if  $H$ is not of real type then
  $\mathcal{G}^\textup{uc}(\Gamma)$ is countably infinite.
\end{pro}

\begin{proof}
  The proof is similar to the proof of
  Proposition~\ref{pro:description_G_Gamma}: one simply replaces
  $[H,H]$ by~$N$.  This is justified by the observation that $G \in
  \mathcal{G}(\Gamma)$ is unipotently connected if and only if
  $\Nil(G) = F = \Nil(H)$, by Corollary~\ref{cor:nil_and_fitt}.  The
  latter is the case if and only if $\sigma_{G \mid H}\colon H
  \rightarrow T$ vanishes on $F= \Nil(H)$.
\end{proof}

If $H$ is of real type then, by the above, $ \mathcal{G}(\Gamma) =
\mathcal{G}^\textup{uc}(\Gamma) =\{ H \}$.  Otherwise we have the
following result.

\begin{pro} \label{pro:all_are_uc} Let $H \in
  \mathcal{G}^\textup{uc}(\Gamma)$, and assume $H$ is not of real
  type. Then the following assertions are equivalent:
 \begin{enumerate}
 \item $\mathcal{G}(\Gamma) = \mathcal{G}^\textup{uc}(\Gamma)$,
 \item $\Nil(H) = [H,H]$,
 \item $[\Gamma,\Gamma]$ has finite index in~$\Fitt(\Gamma)$.
 \end{enumerate}
\end{pro}

\begin{proof}
  The assumption that $H$ is not of real type implies that $T_\text{c}$ is
  not trivial.  Hence $\Hom_{\textup{Lie}}(H/\Gamma [H,H],T) =
  \Hom_{\textup{Lie}}(H/\Gamma \Nil(H),T)$ if and only if $\Nil(H) =
  [H,H]$.  Thus Propositions~\ref{pro:description_G_Gamma} and
  \ref{pro:description_u-connected_G_Gamma} show that (1) and (2) are
  equivalent.

  By Lemma~\ref{lem:commutators}, the group $[\Gamma,\Gamma]$ is a
  lattice in~$[H,H]$, and, by Corollary~\ref{cor:nil_and_fitt}, the
  group $\Fitt(\Gamma)$ is a lattice in~$\Nil(H)$.  Moreover, $[H,H]$
  and $\Nil(H)$ are simply connected.  Since $[\Gamma,\Gamma]
  \subseteq \Fitt(\Gamma)$ and $[H,H] \subseteq \Nil(H)$,
  Lemma~\ref{lem:criterion_lattice} shows that (2) is equivalent
  to~(3).
\end{proof}

We obtain an interesting class of polycyclic groups which admit
Zariski-dense embeddings exclusively into simply connected, solvable
Lie groups which are unipotently connected.

\begin{cor} \label{cor:alluc} Let $\Delta$ be a torsion-free
  polycyclic group satisfying the following conditions:
  \begin{enumerate}
  \item $\Delta \subseteq (A_\Delta)_{\circ}$, i.e., $\Delta$ is
    contained in the identity component of its real algebraic hull,
%  \item $\Fitt(\Delta)/[\Delta,\Delta]$ is the torsion subgroup of
%    $\Delta / [\Delta,\Delta]$.
   \item $\Delta/\Fitt(\Delta)$ is torsion-free,
   \item $[\Delta, \Delta]$ is of finite index in~$\Fitt(\Delta)$.
  \end{enumerate}
  Then $\mathcal{G}(\Delta) = \mathcal{G}^\textup{uc}(\Delta) \neq
  \varnothing$.
\end{cor}

\begin{proof} 
  The proof of Proposition~\ref{pro:exists_unipotently_connected},
  shows that $\mathcal{G}^\textup{uc}(\Delta) \neq \varnothing$, and
  the claim follows from Proposition~\ref{pro:all_are_uc}.
\end{proof}

Recall that the group of Lie automorphisms $\Aut(H)$ of a simply
connected Lie group $H$ is isomorphic to the real algebraic group
$\Aut(\mathfrak{h})$, where $\mathfrak{h}$ is the real Lie algebra
associated to~$H$. Let $H$ be simply connected solvable. Then every automorphism $\phi \in \Aut(H)$ extends uniquely to an $\R$-defined automorphism $\Phi \in
\Aut_\R(\mathbf{A}_H)$ of the $\R$-defined algebraic
group~$\mathbf{A}_H$. 

Define 
$$ \Aut(H)^1 =  \, \ker \,  \left( \,   \Aut(H) \longrightarrow \Aut(H/\Nil(H))\, \right)  \; . $$

\begin{pro}\label{pro:Hinv_implies_Ginv}
  Let $G, H \in \mathcal{G}(\Gamma)$, and suppose that $G$ is
  unipotently connected.  Let $\phi \in \Aut(H)^1$ and let $\Phi
  \in \Aut_\R(\mathbf{A})$ be its extension.  Then
%  \footnote{BK:
%    slightly stronger: $\Phi \vert_G \in \Aut(G)^\circ$; add one
%    sentence explanation.}
     $\Phi(G) = G$.
 \end{pro}

\begin{proof}
   Since $H$ and $G$ are normal in $A$, their nilpotent radicals
  $\Nil(H)$ and $\Nil(G)$ are normal in~$A$.  Note that we have
  $\Nil(H) \subseteq \Nil(G) = F$ by Corollary~\ref{cor:nil_and_fitt},
  where $F$ is the Zariski-closure of $\Fitt(\Gamma)$ in~$A$.
  Clearly, $\Phi(\Nil(H)) = \Nil(H)$ and, since $H$ is Zariski-dense
  in~$A$, we deduce that the induced automorphism $\Phi
  \vert_{A/\Nil(H)}$ satisfies $\Phi \vert_{A/\Nil(H)} =
  \id_{A/\Nil(H)}$.  Furthermore, $\Nil(H) \subseteq F$ implies that
  $\Phi(F) = F$ and that the induced automorphism $\Phi \vert_{A/F}$
  satisfies $\Phi \vert_{A/F} = \id_{A/F}$.  Thus we have
  $\Phi(G) \subseteq G F = G$. 
\end{proof}

\begin{lem} Let $H$ be a simply connected solvable Lie group. 
Then $\Aut(H)^\circ \leq  \Aut(H)^1$. \label{lem:autoinaut1}
\end{lem}
\begin{proof}
 The Lie algebra corresponding to $\Aut(H)^\circ$ is isomorphic to
  the derivation algebra $\Der(\mathfrak{h})$ of the Lie algebra
  $\mathfrak{h}$ associated to~$H$.  It is known that the solvable Lie
  algebra $\mathfrak{h}$ is mapped into its nilradical by every
  derivation of $\mathfrak{h}$; see \cite[\S II.7]{Ja79}.  Thus
  $\Aut(H)^\circ$ acts trivially on~$H/\Nil(H)$.  
  % In particular, we have $\phi \vert_{H / \Nil(H)} = \id_{H / \Nil(H)}$.
\end{proof}

Under the assumptions of
Proposition~\ref{pro:Hinv_implies_Ginv} we
obtain a natural homomorphism $\Aut(H)^\circ \rightarrow
\Aut(G)^\circ$. Since $G$ is Zariski-dense this homomorphism is
injective. The following corollary records that the groups contained
in $\mathcal{G}_{\textup{uc}}(\Gamma)$ are very similar.

\begin{cor} \label{cor:AHo_eq_AGo} Let $G, H \in
  \mathcal{G}_{\textup{uc}}(\Gamma)$.  Then there exists a natural
  isomorphism $\Aut(H)^\circ \rightarrow \Aut(G)^\circ$.
\end{cor}
 
%%%%%

\section{Description of the space $\mathcal{G}_{\Gamma,G}$}
\label{sect:GGammaG} We keep in place the notational conventions of
Section~\ref{sect:GGamma}.  In particular, $\Gamma$ is a torsion-free
polycyclic group whose algebraic hull $\mathbf{A} = \mathbf{A}_\Gamma$
is connected, and $A =\mathbf{A}_\R$ denotes the group of real points.
Throughout we suppose that $\mathcal{G}(\Gamma) \ne \varnothing$.

\smallskip

In addition, we fix $G \in \mathcal{G}(\Gamma)$ and define
$$
\mathcal{G}_{\Gamma,G} = \bigcup \{ \mathcal{G}(\Delta) \mid
\text{$\Delta \subseteq G$ a lattice with $\Delta \cong \Gamma$} \} \; .
$$
This collection of Lie subgroups of $A$ relates to the space
$\mathcal{X}(\Gamma,G)$ via the map
\begin{equation} \label{equ:nat_map} \mathcal{X}(\Gamma,G) \rightarrow
  \mathcal{G}_{\Gamma,G}, \quad \phi \mapsto \Phi(G),
\end{equation}
where $\Phi \in \Aut_\R(\mathbf{A})$ is the unique extension of
$\phi$; see Corollary~\ref{cor:hull_unique}. We note that $\Phi(G) \in
\mathcal{G}(\Delta)$, where $\Delta = \phi(\Gamma)$.

Natural actions of $\Aut(G)$ on $\mathcal{X}(\Gamma,G)$ and on
$\mathcal{G}_{\Gamma,G}$ are given by
$$
\theta \cdot \phi = \theta \circ \phi \quad \text{and} \quad \theta
\cdot H = \Theta(H),
$$
where $\theta \in \Aut(G)$ with unique extension $\Theta \in
\Aut_\R(\mathbf{A})$, $\phi \in \mathcal{X}(\Gamma,G)$ and $H \in
\mathcal{G}_{\Gamma,G}$.  The map~\eqref{equ:nat_map} is
$\Aut(G)$-equivariant with respect to these actions. 

By Lemma \ref{lem:autoinaut1},
$ \Aut(G)^1 = \ker \left( \, \Aut(G) \rightarrow \Aut(G/\Nil(G)) \, \right)$ is of 
finite index in $\Aut(G)$ and define $$ c(G) = \lvert \Aut(G) : \Aut(G)^1 \rvert . $$

\begin{pro} \label{pro:finfibers} Suppose that $G$ is unipotently
  connected.  Then \begin{enumerate}
  \item $\Aut(G)^1$ acts trivially on the image
    $\widetilde{\mathcal{G}}_{\Gamma,G}$ of the map
    \eqref{equ:nat_map}.
  \item the fibres of the induced map
    \begin{equation} \label{equ:nat_map_mod_Auto} \Aut(G)^1
      \backslash \mathcal{X}(\Gamma,G) \rightarrow
      \mathcal{G}_{\Gamma,G}, \quad [\phi]_{ \Aut(G)^1} \mapsto
      \Phi(G) \; .
    \end{equation}
    are finite and bounded in size by $c(G)$.
  \end{enumerate}
\end{pro}

% alternative, more down to earth version of the proof
% keep for the record
%
% \begin{proof}
%   Let $H = \Phi(G)$ be in the image of the map \eqref{equ:nat_map},
%   where $\Phi$ is the extension of $\phi\colon\Gamma \rightarrow G$.
%   Since $G$ is unipotently connected, so is~$H$.  Let $\Theta \in
%   \Aut_{\R}(\mathbf{A})$ be the extension of $\theta \in
%   \Aut(G)_{\circ}$.  Both $G$ and $H$ are elements of $
%   \mathcal{G}(\phi(\Gamma))$, and $H$ is unipotently connected.  Thus
%   Proposition~\ref{pro:Hinv_implies_Ginv} implies that $\theta \cdot H
%   = \Theta(H) = H$, showing~(1).

%   It remains to verify (2).  Let $\psi_1, \psi_2 \in
%   \mathcal{X}(\Gamma,G)$, with extensions $\Psi_1, \Psi_2 \in
%   \Aut_\R(\mathbf{A})$, such that $\Psi_1(G) = \Psi_2(G) = H$.  Then
%   $\Omega_1 = \Psi_1 \circ \Phi^{-1}$ and $\Omega_2 = \Psi_2 \circ
%   \Phi^{-1}$ restrict to elements $\omega_1, \omega_2 \in \Aut(H)$.
%   Because $H \cong G$, we have $\lvert \Aut(H) : \Aut(H)_\circ \rvert
%   = \lvert \Aut(G) : \Aut(G)_\circ \rvert$.  Therefore it suffices to
%   show that $[\psi_1]_{\Aut(G)_\circ} = [\psi_2]_{ \Aut(G)_\circ}$ if
%   $\omega_1 \equiv \omega_2$ modulo $\Aut(H)_\circ$.  Suppose that
%   $\omega_1 \circ \omega_2^{-1} \in \Aut(H)_\circ$.  Since $G,H \in
%   \mathcal{G}_{\textup{uc}}(\phi(\Gamma))$,
%   Proposition~\ref{pro:Hinv_implies_Ginv} implies that $\Psi_1 \circ
%   \Psi_2^{-1} = \Omega_1 \circ \Omega_2^{-1}$ restricts to an element
%   of~$\Aut(G)_\circ$.  Hence $[\psi_1]_{\Aut(G)_\circ} =
%   [\psi_2]_{\Aut(G)_\circ}$.
% \end{proof}

\begin{proof}
  Let $H = \Phi(G)$ be in the image of the map \eqref{equ:nat_map},
  where $\Phi$ is the extension of $\phi\colon\Gamma \rightarrow G$.
  Since $G$ is unipotently connected, so is~$H$.  Let $\Theta \in
  \Aut_{\R}(\mathbf{A})$ be the extension of $\theta \in
  \Aut(G)^1$.  Both $G$ and $H$ are elements of $
  \mathcal{G}(\phi(\Gamma))$, and $H$ is unipotently connected.  Thus
  Proposition~\ref{pro:Hinv_implies_Ginv} implies that $\theta \cdot H =
  \Theta(H) = H$, showing~(1).

  Observe next that composition of maps provides a natural action of
  $\Aut(H)$ on the set
  $$
  \{ \phi \colon \Gamma \hookrightarrow H \mid \phi(\Gamma) \text{ a
    lattice in $H$ and } \Phi(G) = H \},
  $$
  where as before $\Phi \in \Aut_\R(\mathbf{A})$ denotes the unique
  extension of~$\phi$.  This action is transitive: given $\phi, \psi
  \colon \Gamma \rightarrow H$ with extensions $\Phi, \Psi$ such that
  $\Phi(G) = \Psi(G) = H$, the map $\Phi \circ \Psi^{-1}$ restricts to
  an automorphism $\theta \in \Aut(H)$ such that $\phi = \theta \circ
  \psi$.  Furthermore, if $\phi,\psi \in \mathcal{X}(\Gamma,G)$, then
  $G,H \in \mathcal{G}_{\textup{uc}}(\phi(\Gamma))$ and hence
  Proposition~\ref{pro:Hinv_implies_Ginv} shows that, if $\phi$ and $\psi$ lie
  in the same $\Aut(H)^1$-orbit, then $[\phi]_{ \Aut(G)^1} =
  [\psi] _{ \Aut(G)^1}$.  Therefore each fibre of the map
  \eqref{equ:nat_map_mod_Auto} has size bounded by 
  $c(H)$.
   % $\lvert \Aut(H) : \Aut(H)^1 \rvert$.  
%   Since $\Aut(H)$ is a real algebraic group,
%  it has only finitely many connected components. Thus, by Lemma \ref{lem:autoinaut1} the index of $ \Aut(H)^1$ in $\Aut(H)$ is finite. Because $H \cong
%  G$, this proves~(2).
\end{proof}

The following consequence is a key ingredient in the proof of our main
result Theorem~\ref{thm:MainA}.

\begin{cor} \label{cor:HomsToGroups} Let $G$ be a simply connected,
  solvable Lie group which is unipotently connected. Let $\Gamma$ be a
  Zariski-dense lattice in~$G$.  Then the map
%  $$    
%  \Aut(G) \backslash \mathcal{X}(\Gamma,G) \rightarrow \Aut(G)
%  \backslash \mathcal{G}_{\Gamma,G}
%  $$ 
  \begin{equation} \label{equ:nat_map_mod_Aut} \Aut(G)\backslash
    \mathcal{X}(\Gamma,G) \rightarrow \Aut(G) \backslash
    \mathcal{G}_{\Gamma,G}, \quad [\phi] \mapsto [\Phi(G) ] \; ,
  \end{equation}
  induced by \eqref{equ:nat_map}, has finite fibres, bounded in size
  by $c(G)$.
\end{cor}

\begin{proof} As an immediate consequence of
  Proposition~\ref{pro:finfibers} and the fact the the map
  \eqref{equ:nat_map} is $\Aut(G)$-equivariant we obtain that the
  fibres of the map \eqref{equ:nat_map_mod_Aut} are bounded in size by
  $c(G)$.
\end{proof}

%%%%%%%%%%%%%%%%%%%%%%%%%%%%%%

\section{Proofs of Theorem~\ref{thm:MainA}, Theorem~\ref{thm:MainE} and their corollaries} \label{sec:proof_main_results}

In this section we prove all the results from Theorem~\ref{thm:MainA}
up to Corollary~\ref{cor:ToMainA3}, which were stated in the
introduction.  The proofs of Theorem~\ref{thm:MainH} and its
Corollary~\ref{cor:toMainH1} are given in Section~\ref{sec:Chabauty}.

% We keep in place the conventions and notation used in Sections
% \ref{sect:GGamma} and~\ref{sect:GGammaG}.  In particular, the
% polycyclic group $\Gamma$ can be embedded as a Zariski-dense lattice
% into a simply connected, solvable Lie group.

\subsection{Proof of Theorem \ref{thm:MainA}} \label{sec:proof_of_A}
Throughout, let $\Gamma$ be a Zariski-dense lattice in a simply
connected, solvable Lie group~$G$.  We fix an algebraic hull
$\mathbf{A} = \mathbf{A}_{\Gamma}$ of $\Gamma$ such that $\Gamma
\subseteq G \subseteq \mathbf{A}$ and, in this construction,
$\mathbf{A}$ is also an algebraic hull for $G$; see Corollary
\ref{cor:AGamisAG}.  Thus $\mathbf{A}$ is connected and we write
$\mathbf{A} = \mathbf{U} \rtimes \mathbf{T}$, where $\mathbf{U} =
\Rad_\textup{u}(\mathbf{A})$ and the maximal algebraic torus
$\mathbf{T}$ are defined over~$\R$.  Put $A = \mathbf{A}_\R$, $U =
\mathbf{U}_\R$ and $T = \mathbf{T}_\R$. Since $\mathbf{T} = \mathbf{A}/ \mathbf{U}$, there is a natural induced action of $ \Aut_{\R}(\mathbf{A})$ on $\mathbf{T}$. Let 
$$  \Aut_{\R}(\mathbf{A})^1 = \, \ker \, \left( \, 
\Aut_\R(\mathbf{A}) \longrightarrow \Aut_\R(\mathbf{T}) \, \right)$$ be the kernel of the resulting map.
Put 
$$
c(\mathbf{A})  = \lvert \Aut_{\R}(\mathbf{A}) : \Aut_{\R}(\mathbf{A})^1
\rvert
$$
for  the index of  $\Aut_{\R}(\mathbf{A})^1$. 
By the rigidity of algebraic tori (see
  \cite[III.8]{Bo91}), the identity component
  $\Aut_\R(\mathbf{A})^\circ$ is contained in 
  $\Aut_{\R}(\mathbf{A})^1$. Therefore $c(\mathbf{A}) $ is
  finite and bounded by the number of connected components of the real algebraic
group~$\Aut_\R(\mathbf{A})$.  

We also consider the sets
\begin{align*} 
  \mathcal{U}_{\Gamma} & = \{ \Delta U \mid \Delta \text{ a
    Zariski-dense subgroup of $A$ isomorphic to $\Gamma$} \}, \\
  \mathcal{N}_{\Gamma,G} & = \{ \Delta \Nil(G) \mid \Delta \text{ a
    lattice in $G$ isomorphic to $\Gamma$} \}.
\end{align*}
% Note that we may view  $\mathcal{U}_\Gamma$ as a set of subgroups
% of $\mathbf{T} = \mathbf{A} /\mathbf{U}$, and $ \mathcal{N}_{\Gamma,G}$ 
% as a set of subgroups of $G/\Nil(G)$. These sets of subgroups are
% related by the natural homomorphism $G/\Nil(G) \rightarrow \mathbf{T}$. 

\begin{lem} \label{lem:setGamU} The set $\mathcal{U}_\Gamma$ has
  finite size, bounded by~$c(\mathbf{A}) $.
\end{lem}

\begin{proof} Let $\Delta$ be a Zariski-dense subgroup of $A$
  isomorphic to $\Gamma$, and choose an isomorphism $\phi \colon
  \Gamma \rightarrow \Delta$.  Note that $A$ is also a real algebraic
  hull for~$\Delta$.  By the uniqueness of algebraic hulls, $\phi$
  extends to an $\R$-defined algebraic automorphism $\Phi \colon
  \mathbf{A} \rightarrow \mathbf{A}$ so that $\Phi(\Gamma) = \Delta$.

  Every automorphism $\Phi \in \Aut_\R(\mathbf{A})$ induces an
  automorphism of the $\R$-defined torus $\mathbf{A}/\mathbf{U} \cong
  \mathbf{T}$.    Hence, if $\Phi, \Psi \in \Aut_\R(\mathbf{A})$ belong to the same
  coset of $\Aut_\R(\mathbf{A})^1$, then $\Phi (\Gamma) U = \Psi
  (\Gamma) U$.  Hence $\lvert \mathcal{U}_\Gamma \rvert \leq c(\mathbf{A}) $.
\end{proof} 

According to Proposition~\ref{pro:Zdense} and
Corollary~\ref{cor:all_Zdense}, every lattice $\Delta$ in $G$ with
$\Delta \cong \Gamma$ is Zariski-dense in~$A$.  Since $\Nil(G)
\subseteq U$, we obtain a map
\begin{equation}\label{equ:NtoU}
  \mathcal{N}_{\Gamma,G} \rightarrow \mathcal{U}_\Gamma, \quad \Delta
  \Nil(G) \mapsto \Delta U.
\end{equation}

\begin{lem} \label{lem:setGamN} If $G$ is unipotently connected, then
  the map~\eqref{equ:NtoU} embeds $\mathcal{N}_{\Gamma,G}$ into
  $\mathcal{U}_\Gamma$.  In particular, $\mathcal{N}_{\Gamma,G}$ has
  finite size, bounded by~$c(\mathbf{A}) $.
\end{lem}

\begin{proof} 
  The projection $\mathbf{A} \rightarrow \mathbf{T}$ associated to the
  decomposition $\mathbf{A} = \mathbf{U} \rtimes \mathbf{T}$ restricts
  to a homomorphism $\tau \colon G \rightarrow T$ with $\Nil(G)
  \subseteq \ker \tau$.  If $G$ is unipotently connected, then $\ker
  \tau = \Nil(G)$ by Lemma~\ref{lem:Gtau_properties}, and hence the
  induced homomorphism
  $$ 
  G/\Nil(G) \rightarrow T \cong A/U
  $$
  is injective.  Therefore, the map~\eqref{equ:NtoU} is injective, and
  Lemma~\ref{lem:setGamU} can be applied.
\end{proof} 

\begin{lem} \label{lem:cGA} If $G$ is unipotently connected, then $c(G) \leq c(\mathbf{A}) $.
\end{lem}
\begin{proof} 
  Since the homomorphism
  $
  G/\Nil(G)   \rightarrow T \cong A/U
  $
  is injective and has Zariski-dense image
  extension of automorphisms   $$ \Aut(G) /\Aut(G)^1 \, \longrightarrow \, \Aut(A)/ \Aut(A)^1$$ is well defined and an  injective map.
\end{proof}

We reformulate these results in a slightly different way.  Recall that
every lattice $\Delta$ in~$G$ maps to a lattice $ \Delta
\Nil(G)/\Nil(G)$ in the vector group $V= G/\Nil(G)$. Two lattices in
$V$ are called commensurable if they have a common finite index
subgroup.  Equivalently, they are commensurable if and only if they
span the same $\Q$-vector space in~$V$.

\begin{cor} Let $V = G/\Nil(G)$ and let $\eta \colon G \rightarrow V$
  denote the natural projection.  Let $\pi \colon \mathbf{A}
  \rightarrow \mathbf{T}$ denote the projection associated to the
  decomposition $\mathbf{A} = \mathbf{U} \rtimes \mathbf{T}$.  Then
  the following hold.
  \begin{enumerate}
  \item Let $\Delta_{1}, \Delta_{2}$ be lattices in $G$ which are
    isomorphic to~$\Gamma$.  Then their images $\eta(\Delta_{1}),
    \eta(\Delta_{2})$ in $V$ are commensurable if their images
    $\pi(\Delta_1), \pi(\Delta_2)$ in $\mathbf{T}$ are equal.
    \item The set $\mathcal{V}_{\Gamma,G} = \{ \eta(\Delta) \leq V \mid
    \Delta \subseteq G \text{ a lattice with } \Delta \cong \Gamma \}$
    falls into finitely many commensurability classes; in particular,
    it is countable.
  \item If $G$ is unipotently connected, then the set
    $\mathcal{V}_{\Gamma,G}$ is finite.
  \end{enumerate}
\end{cor}

\begin{proof} 
  Let $\Delta$ be a lattice in $G$ which is isomorphic to~$\Gamma$.
  Proposition~\ref{pro:fittislattice} shows that $\Fitt(\Delta)$ is
  cocompact in $\up(G)$, hence $\eta(\Fitt(\Delta))$ and
  $\eta(\up(G))$ are commensurable subgroups of~$V$.  Therefore
  $\eta(\Delta)$ and $\eta(\Delta \up(G))$ are commensurable.  Since
  $\up(G) = G \cap \ker \pi$, this implies~(1).

  By Lemma~\ref{lem:setGamU}, the set $\{ \pi(\Delta) \mid \Delta
  \subseteq G \text{ a lattice with } \Delta \cong \Gamma \}$ which
  naturally embeds into $\mathcal{U}_\Gamma$ is finite.  Therefore, by
  (1), the set $\mathcal{V}_{\Gamma,G}$ consists of finitely many
  commensurability classes.  Hence (2) holds.

  Clearly, the set $\mathcal{V}_{\Gamma,G}$ admits a one-to-one
  correspondence to the set~$\mathcal{N}_{\Gamma,G}$. Thus (3) is a
  direct consequence of Lemma~\ref{lem:setGamN}.
\end{proof}

Next we consider the set
\begin{equation*}
  \widetilde{\mathcal{G}}_{\Gamma,G} = \{ \Phi(G) \mid \Phi \in
  \Aut_\R(\mathbf{A}) \text{ such that } \Phi(\Gamma) \subseteq G \}
\end{equation*}
of subgroups of $A$, viz.\ the image of the map described
in~\eqref{equ:nat_map}.  We are also interested in its subset
$$
\widetilde{\mathcal{G}}_{\Gamma,G}^\textup{Nil} = \{ H \in
\widetilde{\mathcal{G}}_{\Gamma,G} \mid \Nil(H) = \Nil(G) \} \; .
$$

\begin{pro} \label{pro:setofPhiG_N} If $G$ is unipotently connected,
  then $\widetilde{\mathcal{G}}_{\Gamma,G}^\textup{Nil}$ is finite,
  bounded in size by~$c(\mathbf{A}) ^2$.
\end{pro}

\begin{proof}
  Suppose that $G$ is unipotently connected.  By
  Corollary~\ref{cor:nil_and_fitt}, the groups $\Fitt(\Gamma)$ and
  $\Nil(G)$ have the same Zariski-closure $\mathbf{F} =
  \ac{\Fitt(\Gamma)} = \ac{\Nil(G)}$ in $\mathbf{A}$, and $F =
  \mathbf{F}_\R = \Nil(G)$.  By Lemma~\ref{lem:setGamN} the set
  $\mathcal{N}_{\Gamma,G}$ is finite, of size bounded by~$c(\mathbf{A}) $.
  Hence it suffices to show that the set 
   \begin{equation} \label{eqn:orbit1}
    \{ \Phi(G) \mid \Phi \in \Aut_\R(\mathbf{A}) \text{ such that }
    \Phi(\Gamma F) = \Gamma F \text{ and } \Phi(F) = F \}
  \end{equation}
  is finite, of size bounded by~$c(\mathbf{A}) $.

  The collection of all $\Phi \in \Aut_\R(\mathbf{A})$ with
  $\Phi(\Gamma F) = \Gamma F$ and $\Phi(F) = F$ forms a subgroup
  $\Aut_{\R}(\mathbf{A},\Gamma F)$ of~$\Aut_\R(\mathbf{A})$.  Let
  $\Phi \in \Aut_{\R}(\mathbf{A}, \Gamma F) \cap
  \Aut_{\R}(\mathbf{A})^1$.  Then $\Phi$
  induces the identity morphism on~$\mathbf{A}/\mathbf{U}$.  We
  observe that
  $$
  \Gamma U/U \cong \Gamma/(\Gamma \cap U) = \Gamma/\Fitt(\Gamma) =
  \Gamma/(\Gamma \cap F) \cong \Gamma F/F
  $$
  and that this chain is $\Phi$-equivariant.  Since $\Phi$ acts
  trivially on $\Gamma U / U \subseteq A/U$, it acts trivially
  on~$\Gamma F/F$.  Since $\Gamma$ is Zariski-dense in $\mathbf{A}$,
  this shows that $\Phi$ induces the identity morphism
  on~$\mathbf{A}/\mathbf{F}$.   We infer that $\Phi(G) \subseteq G F =G$.  This shows
  that the set \eqref{eqn:orbit1} is finite and bounded in size
  by~$c(\mathbf{A}) $.
\end{proof}

\begin{lem} \label{lem:N=G_with_tilde} If $G$ is unipotently
  connected, then $\widetilde{\mathcal{G}}_{\Gamma,G}^\textup{Nil} =
  \widetilde{\mathcal{G}}_{\Gamma,G}$.
\end{lem}

\begin{proof}
  Suppose that $G$ is unipotently connected.  We need to show that
  $\widetilde{\mathcal{G}}_{\Gamma,G}^\textup{Nil} \supseteq
  \widetilde{\mathcal{G}}_{\Gamma,G}$.  Let $H \in
  \widetilde{\mathcal{G}}_{\Gamma,G}$.  Then $H = \Phi(G)$ for $\Phi
  \in \Aut_\R(\mathbf{A})$ such that $\Delta = \Phi(\Gamma)$ is a
  lattice in $G$ and in~$H$.  Recall that $\mathbf{A}$ is an algebraic
  hull of~$G$.  Thus $\mathbf{A}$ is an algebraic hull also of
  $\Delta$ and of~$H$.  We observe that $\Phi(G \cap U) = H \cap U$.
  Since $G$ is unipotently connected, so is~$H$.  Thus
  Proposition~\ref{pro:nil=nil}, applied to $G,H \in
  \mathcal{G}(\Delta)$, yields $\Nil(H) = \Nil(G)$.
\end{proof}
 
From Proposition~\ref{pro:setofPhiG_N} and
Lemma~\ref{lem:N=G_with_tilde} we deduce the following.
 
\begin{cor} \label{cor:setofPhiG} If $G$ is unipotently connected,
  then $\widetilde{\mathcal{G}}_{\Gamma,G}$ is finite, and bounded in size
  by~$c(\mathbf{A}) ^2$.
\end{cor}

% \begin{proof}
%   By the proof of ,
%   $\widetilde{\mathcal{G}}_{\Gamma,G}^\textup{Nil}$ is a finite union
%   of orbits of subgroups $\Aut_{\R}(\mathbf{A}, \Delta F)$, and the
%   number of elements in each orbit is bounded by $c(G)$.  By
%   Lemma~\ref{lem:N=G_with_tilde} and Lemma~\ref{lem:setGamU} there are
%   at most $c(G)$ such orbits. Therefore,
%   $\widetilde{\mathcal{G}}_{\Gamma,G}$ is finite, bounded in size by
%   $c(G)^2$. By Lemma \ref{lem:N=G_with_tilde},
%   $\widetilde{\mathcal{G}}_{\Gamma,G}^\textup{Nil} =
%   \widetilde{\mathcal{G}}_{\Gamma,G}$.
% \end{proof}
\medskip 

\begin{proof}[Proof of Theorem~\ref{thm:MainA}]
  By Corollary~\ref{cor:HomsToGroups}, the natural map
  \eqref{equ:nat_map_mod_Auto} from $ \Aut(G)^1 \backslash
  \mathcal{X}(\Gamma,G)$ onto $\widetilde{\mathcal{G}}_{\Gamma,G}$ has
  finite fibres, bounded in size by~$c(G)$.  By
  Corollary~\ref{cor:setofPhiG}, the image of
  \eqref{equ:nat_map_mod_Auto} has at most $c(\mathbf{A}) ^2$ elements.  Hence,
  $\Aut(G)^1\backslash \mathcal{X}(\Gamma,G)$ and therefore
  also its quotient $\Aut(G) \backslash \mathcal{X}(\Gamma,G)$ have at
  most $c(\mathbf{A}) ^2 c(G)$ elements. 
\end{proof}

\begin{pro} The constants $c(G)$ and $c(\mathbf{A}) $ are bounded by a constant which depends
  only on~$\dim \Nil(G)$.
\end{pro}

\begin{proof} By Lemma \ref{lem:cGA}, $c(G) \leq c(\mathbf{A}) $.
By definition, $c(\mathbf{A}) $  is the cardinality of a finite subgroup of the
  algebraic automorphism group of a maximal torus $\mathbf{T}$
  in~$\mathbf{A}$.  The algebraic automorphism group of $\mathbf{T}$
  is isomorphic to the group of integral matrices $\GL(n,\Z)$, where
  $n = \dim T$, see~\cite[Chapter~III.8]{Bo91}.  The cardinality of a
  finite subgroup in $\GL(n,\Z)$ is bounded by a constant depending on
  $n$ only, as is known classically~\cite{Min1887}.  Observe further
  that $\mathbf{T}$ admits a faithful representation on the
  complexification of the Lie algebra of the nilradical of~$G$.
  Therefore, $n= \dim \mathbf{T}$ can be bounded in terms of~$\dim
  \Nil(G)$.
\end{proof}

%%%

\subsection{Proofs of Corollaries \ref{cor:MainB}, \ref{cor:toMainA1}
  and \ref{cor:toMainA2}} \label{sect:proofcorB} Throughout, let
$\Gamma$ be a torsion-free polycyclic group with algebraic hull
$\mathbf{A} = \mathbf{A}_\Gamma$.  The Fitting subgroup
$\Fitt(\Gamma)$ is a characteristic subgroup of the polycyclic
group~$\Gamma$.  We set
$$ 
\Aut^\circ(\Gamma) = \Cen_{\Aut(\Gamma)}(\Gamma/\Fitt(\Gamma)),
$$
i.e., $\Aut^\circ(\Gamma)$ is the group of all automorphisms of
$\Gamma$ which induce the identity on the Fitting quotient $\Gamma /
\Fitt(\Gamma)$.

\begin{lem} \label{lem:AutNull} Let $\Gamma$ be a lattice in a simply
  connected, solvable Lie group~$G$.  Then the group
  $\Aut^\circ(\Gamma)$ has finite index in~$\Aut(\Gamma)$.
\end{lem}

\begin{proof} Since every $\phi \in \Aut(\Gamma)$ extends to an
  automorphism $\Phi \in \Aut_\R(\mathbf{A})$, we obtain an embedding
  $\Aut(\Gamma) \hookrightarrow \Aut_\R(\mathbf{A})$.  Since $\lvert
  \Aut_\R(\mathbf{A}) : \Aut_\R(\mathbf{A})^\circ \rvert < \infty$, it
  will be enough to show that
  $$
  \Aut^\circ(\Gamma) \supseteq \{ \phi \in \Aut(\Gamma) \mid \Phi \in
  \Aut_\R(\mathbf{A})^\circ \}.
  $$
  
  Let $\phi \in \Aut(\Gamma)$ with extension $\Phi \in
  \Aut_\R(\mathbf{A})^\circ$.  By the rigidity of tori, $\Phi$ induces
  the identity on $\mathbf{A}/\mathbf{U}$, and by
  Proposition~\ref{pro:fittislattice} we have $\Fitt(\Gamma) = \Gamma
  \cap \mathbf{U}$.  Hence $\phi$ induces the identity on $\Gamma /
  \Fitt(\Gamma)$.

  Alternatively, the claim can be derived as follows.  Let ${\rm
    Inn}(\Gamma)$ denote the group of inner automorphisms of~$\Gamma$.
  By~\cite[Theorem~1.3]{BaGr06}, the group ${\rm Inn}(\Gamma)
  \Aut^\circ(\Gamma)$ is of finite index in $\Aut(\Gamma)$, for any
  polycyclic group~$\Gamma$. Now since $\Gamma$ is a lattice in $G$,
  we deduce that $[\Gamma, \Gamma] \leq \Gamma \cap [G,G] \leq \Gamma
  \cap \Nil(G) \leq \Fitt(\Gamma)$.  This implies ${\rm Inn}(\Gamma)$
  is contained in~$\Aut^\circ(\Gamma)$.
\end{proof}

\begin{lem} \label{lem:ext_Aut0} Every $\phi \in \Aut^\circ(\Gamma)$
  extends uniquely to a $\Q$-defined automorphism $\Phi$ of the
  algebraic group $\mathbf{A}$, and this extension satisfies
  $\Phi_{\mathbf{A}/\mathbf{F}} = \id_{\mathbf{A}/\mathbf{F}}$, where
  $\mathbf{F} = \ac{\Fitt(\Gamma)}$ is the Zariski-closure
  in~$\mathbf{A}$.
\end{lem}

\begin{proof}
  Let $\phi \in \Aut^\circ(\Gamma)$.  Since $\mathbf{A}$ is an
  algebraic hull of $\Gamma$, there is a unique extension $\Phi \in
  \Aut(\mathbf{A})$ which is $\Q$-defined.  Consider the natural
  projection $\mathbf{A} \rightarrow \mathbf{A}/\mathbf{F}$.  Since
  $\Gamma$ is Zariski-dense in $\mathbf{A}$, its image $\Gamma
  \mathbf{F}/\mathbf{F}$ is Zariski-dense in~$\mathbf{A}/\mathbf{F}$.
%   According to our setup, we have $\Gamma \cap \mathbf{F} = \Gamma
%   \cap \Nil(G) = \Fitt(\Gamma)$; see Proposition~{\bf fill in}).
  Clearly, we have $\Fitt(\Gamma) = \Gamma \cap \mathbf{F}$.  Since
  $\phi \in \Aut^\circ(\Gamma)$, this implies that
  $\Phi_{\mathbf{A}/\mathbf{F}}$ acts as the identity on $\Gamma
  \mathbf{F}/\mathbf{F}$ and hence $\Phi_{\mathbf{A}/\mathbf{F}} =
  \id_{\mathbf{A}/\mathbf{F}}$.
\end{proof}

\begin{proof}[Proof of Corollary~\ref{cor:MainB}]
  Let $\Gamma$ be a Zariski-dense lattice in $G$, where $G$ is
  unipotently connected.  We claim that every $\phi \in
  \Aut^\circ(\Gamma)$ extends to an automorphism of~$G$.

  Let $\phi \in \Aut^\circ(\Gamma)$.  By Corollary~\ref{cor:AGamisAG}
  we may assume that $G$ is contained in~$\mathbf{A}$.  By Lemma
  \ref{lem:ext_Aut0}, the extension $\Phi \in \Aut_{\R}(\mathbf{A})$
  induces the identity on $\mathbf{A}/\mathbf{F}$, where $\mathbf{F} =
  \ac{\Fitt(\Gamma)}$.  Since $G$ is unipotently connected, we have
  $\mathbf{F}_\R = \Nil(G)$ according to Corollary
  \ref{cor:nil_and_fitt}.  This implies that $\Phi$ induces the
  identity on~$\mathbf{A}/\ac{\Nil(G)}$.  Therefore, we 
  have $\Phi(G) \subseteq G \Nil(G) = G$.
\end{proof}

As explained in the introduction, Corollaries~\ref{cor:toMainA1}
and~\ref{cor:toMainA2} are direct consequences of
Corollary~\ref{cor:MainB},
Proposition~\ref{pro:exists_unipotently_connected} and
Proposition~\ref{pro:all_are_uc}.

%%%

\subsection{One-to-one correspondence between
  $\mathcal{S}^\mathrm{Z}(\Gamma)$ and $\mathcal{G}(\Gamma)$}
\label{sect:Structure_set}
Let $\Gamma$ be a Zariski-dense lattice in a simply connected,
solvable Lie group.  Recall from the introduction that the
\emph{structure set} $\mathcal{S}^\textup{Z}(\Gamma)$ consists of
equivalence classes $[\phi]_{{\mathcal{S}^\textup{Z}(\Gamma)}}$ of
embeddings
$$ 
\phi\colon \Gamma \hookrightarrow G_\phi
$$ 
of $\Gamma$ as a Zariski-dense lattice into simply connected, solvable
Lie groups.  Two embeddings $\phi \colon \Gamma \hookrightarrow
G_\phi$ and $\psi \colon \Gamma \hookrightarrow G_\psi$ represent the
same element in $\mathcal{S}^\textup{Z}(\Gamma)$, if there exists an
isomorphism of Lie groups $\theta \colon G_\phi \rightarrow G_\psi$
such that $\theta \circ \phi = \psi$.

By Corollary~\ref{cor:AGamisAG}, every $\phi \colon \Gamma
\hookrightarrow G_\phi$ as above admits a unique extension $\Phi
\colon \mathbf{A}_\Gamma \rightarrow \mathbf{A}_{G_\phi}$ to the level
of algebraic hulls, yielding an isomorphism of algebraic groups.  It
is easy to verify that this yields a map $\phi \mapsto
\Phi^{-1}(G_\phi)$ into $\mathcal{G}(\Gamma)$ which is constant on
equivalence classes $[\phi]_{{\mathcal{S}^\textup{Z}(\Gamma)}}$.  Thus
we obtain the following \emph{structure map} for Zariski-dense lattice
embeddings of~$\Gamma$:
\begin{equation*} \label{equ:structure_map} \epsilon \colon
  \mathcal{S}^\textup{Z}(\Gamma) \rightarrow \mathcal{G}(\Gamma),
  \quad [\phi]_{{\mathcal{S}^\textup{Z}(\Gamma)}} \mapsto
  \Phi^{-1}(G_\phi) \; .
\end{equation*}

Let $G$ be a simply connected, solvable Lie group which contains
$\Gamma$ as a Zariski-dense lattice.  Then the deformation space
$\mathcal{D}(\Gamma,G) = \Aut(G) \backslash \mathcal{X}(\Gamma,G)$
admits a natural embedding
$$
\mathcal{D}(\Gamma,G) \hookrightarrow \mathcal{S}^\textup{Z}(\Gamma),
\quad [\phi]_{\Aut(G)} \mapsto [\phi]_{\mathcal{S}^\textup{Z}(\Gamma)}
$$
into the structure set~$\mathcal{S}^\textup{Z}(\Gamma)$.  We observe
that the image of the embedded $\mathcal{D}(\Gamma,G)$ under
$\epsilon$ is the set
$$ 
\mathcal{G}(\Gamma)_G = \{ H \in \mathcal{G}(\Gamma) \mid H
\cong G \};
$$
under the name $S(G,\Gamma)$ the latter set plays a central role
in~\cite{St94}.

We summarise these facts as follows.

\begin{pro} \label{pro:real_structureset} Let $\Gamma$ be a
  Zariski-dense lattice in a simply connected, solvable Lie group~$G$.
  Then the structure map 
  $$ 
  \epsilon \colon \mathcal{S}^\textup{Z}(\Gamma) \rightarrow
  \mathcal{G}(\Gamma)
  $$ 
  is a bijection which maps the deformation space
  $\mathcal{D}(\Gamma,G) \subseteq \mathcal{S}^\textup{Z}(\Gamma)$
  onto the subset $\mathcal{G}(\Gamma)_G \subseteq
  \mathcal{G}(\Gamma)$.
\end{pro}

\begin{proof} 
  For every $H \in \mathcal{G}(\Gamma)$ the associated inclusion map
  $\iota \colon \Gamma \hookrightarrow H$ defines an element
  $\delta(H) = [\iota]_{\mathcal{S}^\textup{Z}(\Gamma)}$.  It is
  straightforward to check that the resulting map $\delta \colon
  \mathcal{G}(\Gamma) \rightarrow \mathcal{S}^\textup{Z}(\Gamma)$ and
  $\epsilon$ are mutually inverse to each other.  One also verifies
  easily that $\delta$ maps $\mathcal{G}(\Gamma)_G$ to the subspace
  $\mathcal{D}(\Gamma,G)$ of $\mathcal{S}^\textup{Z}(\Gamma)$.
\end{proof}

%%%

\subsection{Proofs of Theorem \ref{thm:MainE} and Corollaries
  \ref{cor:toMainE1}, \ref{cor:at_most_countable},
  \ref{cor:ToMainA3}.}
Let $\Gamma$ be a Zariski-dense lattice in a simply connected,
solvable Lie group.  By Proposition~\ref{pro:real_structureset}, the
structure set $\mathcal{S}^\textup{Z}(\Gamma)$ can be represented by
the set $\mathcal{G}(\Gamma)$ of subgroups of~$A_\Gamma$.  Hence the
results of Section~\ref{sect:GGamma} lead to several applications
on~$\mathcal{S}^\textup{Z}(\Gamma)$.

\begin{proof}[Proof of Theorem~\ref{thm:MainE}]
  By Proposition~\ref{pro:real_structureset} the cardinality of
  $\mathcal{S}^\textup{Z}(\Gamma)$ is the same as the cardinality
  of~$\mathcal{G}(\Gamma)$.  Therefore, Theorem \ref{thm:MainE}
  follows directly from Proposition~\ref{pro:description_G_Gamma} and
  Corollary~\ref{cor:charac_real_type}.
\end{proof}

Corollary~\ref{cor:at_most_countable} is a direct consequence of
Theorem~\ref{thm:MainE} and Proposition~\ref{pro:real_structureset}.

\begin{proof}[Proof of Corollary~\ref{cor:toMainE1}]
  Suppose that $\Gamma$ is a Zariski-dense lattice in $G$ which is not
  strongly rigid.  Then $\mathcal{S}^\textup{Z}(\Gamma)$ has more than
  one element. Moreover, since $G$ is unipotently connected,
  Proposition~\ref{pro:description_u-connected_G_Gamma} shows that the
  subset $\mathcal{G}_{\text{uc}}(\Gamma)$ of unipotent subgroups in $
  \mathcal{G}(\Gamma)$ is countably infinite.  Using the
  identification of $\mathcal{D}(H,G)$ with $\mathcal{G}(\Gamma)_H
  \subseteq \mathcal{G}(\Gamma)$, Theorem~\ref{thm:MainA} shows that,
  for each $H \in \mathcal{G}_{\text{uc}}(\Gamma)$, the set
  $\mathcal{G}(\Gamma)_H$ is finite.  Since the subset
  $\mathcal{G}_{\text{uc}}(\Gamma)$ is infinite, we conclude that
  there are infinitely many pairwise non-isomorphic unipotently
  connected groups which are contained
  in~$\mathcal{G}_{\text{uc}}(\Gamma)$.  In particular, $\Gamma$ is a
  Zariski-dense lattice in countably infinitely many, pairwise
  non-isomorphic, unipotently connected groups.
\end{proof} 

Corollary~\ref{cor:ToMainA3} is a direct consequence of
Theorem~\ref{thm:MainA}.

%%%%% 

\section{The topologies on $\mathcal{S}^\textup{Z}(\Gamma)$ and
  $\mathcal{D}(\Gamma,G)$} \label{sec:Chabauty}

The purpose of this section is to prove Theorem~\ref{thm:MainH} and
its Corollary~\ref{cor:toMainH1}.  Let $\Gamma$ be a Zariski-dense
lattice in a simply connected, solvable Lie group.  We want to use the
bijection $\epsilon \colon \mathcal{S}^\textup{Z}(\Gamma) \rightarrow
\mathcal{G}(\Gamma)$, described in
Proposition~\ref{pro:real_structureset}, to define a topology
on~$\mathcal{S}^\textup{Z}(\Gamma)$.

For this we briefly recall the definition of two natural topologies on
the collection $\mathfrak{C}_X$ of non-empty closed subsets of a
topological space~$X$.  For any subset $S \subseteq X$ we
put $$
\mathcal{B}_S = \{ C \in \mathfrak{C}_X \mid C \subseteq S \} \quad
\text{and} \quad \mathcal{B}'_S = \{ C \in \mathfrak{C}_X \mid C \cap
S \neq \varnothing \}.
$$
A base for the Vietoris topology on $\mathfrak{C}_X$ is given by all
finite intersections of sets taking the form $\mathcal{B}_U$ or
$\mathcal{B}'_U$, where $U$ is any open subset of~$X$.  The Chabauty
topology on $\mathfrak{C}_X$ is given by all finite intersections of
sets taking the form $\mathcal{B}_{X \setminus K}$ or
$\mathcal{B}'_U$, where $K$ is any compact subset $K$ and $U$ any open
subset of~$X$.
% For any open subset $U \subseteq X$ we put
% $$
% \mathcal{B}_U = \{ C \in \mathfrak{C}_X \mid C \subseteq S \} \quad
% \text{and} \quad \mathcal{B}'_U = \{ C \in \mathfrak{C}_X \mid C \cap
% S \neq \varnothing \}.
% $$
% A base for the Vietoris topology on $\mathfrak{C}_X$ is given by all
% finite intersections of sets taking the form $\mathcal{B}_U$
% or~$\mathcal{B}'_U$.  
% where $U$ is any open subset of~$X$.
% The Chabauty topology on $\mathfrak{C}_X$ is given by all finite
% intersections of sets taking the form $\mathcal{B}_{X \setminus K}$ or
% $\mathcal{B}'_U$, where $K$ is any compact subset $K$ and $U$ any open
% subset of~$X$.
Clearly, if $X$ is Hausdorff, then every Chabauty-open subset of
$\mathfrak{C}_X$ is also Vietoris-open.  If $X$ is compact, then the
Vietoris and the Chabauty topology coincide.  The Chabauty topology
plays an important role in investigating the set $\mathcal{C}(G)$ of
closed subgroups of a locally compact group~$G$; for instance,
see~\cite{MaSi75} for a general discussion.
% The Vietoris and the Chabauty topology can be used to study the
% space of closed subgroups of a topological group, or relevant
% subspaces.  For instance, the Chabauty topology plays an important
% role in investigating the set $\mathcal{C}(G)$ of closed subgroups
% of a locally compact group~$G$; see~\cite{MaSi75} for a general
% discussion\footnote{is this a good reference?}  and \cite{Ha11} for
% a recent contribution to the Chabauty conjecture.\footnote{looks
%   interesting \ldots}

Returning attention to the lattice $\Gamma$ and its structure set
$\mathcal{S}^\textup{Z}(\Gamma)$, let $A = A_\Gamma$ be the real
algebraic hull of~$\Gamma$.  We recall from Section~\ref{sect:GGamma}
that $\mathcal{G}(\Gamma)$ consists of all tight Lie subgroups of $A$
containing~$\Gamma$.  Thus $\mathcal{G}(\Gamma) \subseteq
\mathcal{C}(A)$ can be equipped with the Vietoris or the Chabauty
subspace topology.  It is natural to use the bijection $\epsilon
\colon \mathcal{S}^\textup{Z}(\Gamma) \rightarrow \mathcal{G}(G)$ to
transfer these topologies from $\mathcal{G}(G)$ to
$\mathcal{S}^\textup{Z}(\Gamma)$, and it turns out that both yield the
discrete topology.

\begin{pro}\label{pro:chabauty=discrete} 
  Let $\Gamma$ be a Zariski-dense lattice in a simply connected,
  solvable Lie group.  Then the structure set
  $\mathcal{S}^\textup{Z}(\Gamma)$ is discrete with respect to the
  Chabauty topology inherited from~$\mathcal{G}(\Gamma)$.  The same
  holds for the Vietoris topology.
\end{pro}

\begin{proof}
  Since $A = A_\Gamma$ is Hausdorff, it suffices to consider the
  Chabauty topology.  In fact we will be working with the
  topology on $\mathcal{C}(A)$ which admits as a base the
  collection of finite intersections of sets of the form
  $\mathcal{B}'_U$, where $U$ is an open subset of~$A$.

  Fix a point $H \in \mathcal{G}(\Gamma)$.  We are to show that $H$ is
  an isolated point.  The group $A$ decomposes into a semidirect
  product $A = H \rtimes T$, as described in
  Lemma~\ref{lem:Gtau_properties}.  Let $\tau \colon A \rightarrow T$
  be the associated projection with kernel~$H$.
  Proposition~\ref{pro:tight_Lie_subs_relative} shows how $\hat \tau
  \colon \mathcal{C}(A) \rightarrow \mathcal{C}(T)$, $G \mapsto
  \tau(G)$ maps the set of tight Lie subgroups of $A$
  to~$\mathcal{C}(T)$.  Moreover, a tight Lie subgroup $G$ satisfies
  $\hat \tau(G) = \{1\}$ if and only if $G = H$.  Recall that a Lie
  group does not have small subgroups, i.e.\ there is an open
  neighbourhood of the identity which contains no Lie subgroups except
  the trivial group.  Let $V$ be such a neighbourhood in the Lie
  group~$T$.  Then $U = \tau^{-1}(V) \setminus \{1\}$ is an open
  subset of $A$ and $\mathcal{B}'_U \cap \mathcal{G}(\Gamma) = \{ H
  \}$.  Hence $H$ is isolated in~$\mathcal{G}(\Gamma)$.
\end{proof}

Let $G$ be a simply connected, solvable Lie group which contains
$\Gamma$ as a Zariski-dense lattice.  As described in
Section~\ref{sect:Structure_set} the deformation space
$\mathcal{D}(\Gamma,G) = \Aut(G) \backslash \mathcal{X}(\Gamma,G)$
admits a natural embedding into the structure
set~$\mathcal{S}^\textup{Z}(\Gamma)$.  The space
$\mathcal{X}(\Gamma,G)$ carries the topology of pointwise convergence
and induces the quotient topology on~$\mathcal{D}(\Gamma,G)$.  For
$\mathcal{D}(\Gamma,G) \hookrightarrow \mathcal{S}^\textup{Z}(\Gamma)$
to be continuous, we require that $\mathcal{D}(\Gamma,G)$ is discrete.
We use a result of Wang to show that this is indeed the case.

\begin{pro} \label{pro:Deform_discrete} Let $\Gamma$ be a
  Zariski-dense lattice in a simply connected, solvable Lie group~$G$.
  Then the deformation space $\mathcal{D}(\Gamma,G)$ is discrete with
  respect to the quotient topology inherited
  from~$\mathcal{X}(\Gamma,G)$.
\end{pro}

\begin{proof}
  Let $\phi_0 \in \mathcal{X}(\Gamma,G)$.  We have to show that the
  $\Aut(G)$-orbit of $\phi_0$ in $\mathcal{X}(\Gamma,G)$ is open.
  Without loss of generality we may assume that $\phi_{0} =
  \id_{\Gamma}\colon \Gamma \rightarrow G$ is the identity map
  on~$\Gamma$.  We have $\Gamma \subseteq G \subseteq A =
  \mathbf{A}_\R$, where $\mathbf{A} = \mathbf{A}_G =
  \mathbf{A}_\Gamma$ denotes a common algebraic hull for the groups
  $\Gamma$ and~$G$.  Setting $N = \Nil(G)$, let $D = \Aut(\Gamma
  N)_\circ$ denote the identity component of the group of all Lie
  automorphisms of the group~$\Gamma N$.  In~\cite{Wa63} Wang proved
  that the restriction map
  $$
  D \rightarrow \mathcal{X}(\Gamma,G), \quad \theta \mapsto \theta
  \vert_\Gamma = \theta \circ \phi_{0}
  $$
  yields a homeomorphism between $D$ and $\mathcal{X}(\Gamma,G)_0$,
  the connected component of $\mathcal{X}(\Gamma,G)$
  containing~$\phi_{0}$.

  Every $\phi \in \mathcal{X}(\Gamma,G)$ extends uniquely to an
  element $\Phi \in \Aut_\R(\mathbf{A})$.  Hence we obtain a map $D
  \rightarrow \Aut_\R(\mathbf{A})$ given by $\theta \mapsto \Theta$,
  where $\Theta$ extends~$\theta \vert_\Gamma$.  We contend that every
  $\Theta$ arising in this way restricts to an automorphism of~$G$.
  Let $\theta \in D$ and consider the extension $\Theta$ of~$\theta
  \vert_\Gamma$.  Since $D$ is connected, $\theta$ acts trivially on
  the discrete space~$\Gamma N/N$.  Since $\Gamma$ is Zariski-dense in
  $\mathbf{A}$, this implies that $\Theta$ acts trivially on
  $\mathbf{A}/\mathbf{N}$, where $\mathbf{N}$ denotes the
  Zariski-closure of $N$ in~$\mathbf{A}$.  Thus $\Theta(G) \subset G N =G$.

  This shows that the $\Aut(G)$-orbit of $\phi_0$ in
  $\mathcal{X}(\Gamma,G)$ contains the entire
  component~$\mathcal{X}(\Gamma,G)_0$.  Since $\mathcal{X}(\Gamma,G)$
  is locally path connected, $\mathcal{X}(\Gamma,G)_0$ is open
  in~$\mathcal{X}(\Gamma,G)$.  Hence we conclude that the
  $\Aut(G)$-orbit of $\phi_{0}$ in $\mathcal{X}(\Gamma,G)$ is open.
\end{proof}

Theorem~\ref{thm:MainH} is a direct consequence of
Propositions~\ref{pro:chabauty=discrete}
and~\ref{pro:Deform_discrete}.  We also record the following
conclusion from the proof of Proposition~\ref{pro:Deform_discrete}.

\begin{cor} \label{cor:defdiscrete} Under the same assumptions as in
  Proposition~\ref{pro:Deform_discrete}, the $\Aut(G)$-orbits on
  $\mathcal{X}(\Gamma,G)$ are unions of connected components of
  $\mathcal{X}(\Gamma,G)$.  In particular, they are open.
\end{cor}

The work of Wang~\cite{Wa63} implies that, for any lattice $\Gamma$ in a simply
connected, solvable Lie group the connected components of
$\mathcal{X}(\Gamma,G)$ are manifolds.  If a Lie group $H$ acts
transitively on a connected manifold $Y$ then its identity component
$H_\circ$ also acts transitively on~$Y$. Thus
Corollary~\ref{cor:toMainH1} follows from Corollary~\ref{cor:defdiscrete}.


\begin{thebibliography}{99}

\bibitem{Au73} L.~Auslander, An exposition of the structure of
  solvmanifolds. I. Algebraic theory, Bull.\ Amer.\ Math.\ Soc.\
  \textbf{79} (1973), 262--285.
% 
\bibitem{Ba04} O.~Baues, Infra-solvmanifolds and rigidity of subgroups
  in solvable linear algebraic groups, Topology \textbf{43} (2004),
  903--924.
%  
\bibitem{Ba08} O.~Baues, Deformation spaces for affine
  crystallographic groups, Cohomology of groups and algebraic
  $K$-theory, 55--129, Adv.\ Lect.\ Math.\ (ALM) \textbf{12},
  Int.\ Press, Somerville, MA, 2010.
%  
\bibitem{BaGr06} O.~Baues, F.~Grunewald, Automorphism groups of
  polycyclic-by-finite groups and arithmetic groups,
  Publ.\ Math.\ Inst.\ Hautes \'Etudes Sci.\ No.\ \textbf{104} (2006),
  213--268.
%  
\bibitem{Bo91} A.~Borel, Linear algebraic groups, Second edition,
  Graduate Texts in Mathematics \textbf{126}, Springer-Verlag,
  Berlin-Heidelberg-New York, 1991.
 %
\bibitem{Br94} K.~Brown, Cohomology of groups, Graduate Texts in
  Mathematics \textbf{87}, Springer-Verlag, New York, 1994.
%  
\bibitem{FrKl} R.\ Fricke, F.\ Klein, Vorlesungen \"uber die Theorie
  der automorphen Funktionen, B\"ande I und II, B.\ G.\ Teubner,
  Leibzig, 1897 and 1912.
%
\bibitem{Go88} W.M.~Goldman, Topological components of spaces of
  representations, Invent.\ Math.\ \textbf{93}, 3, 557--607 (1988).
%
\bibitem{Gr70} K.W.~Gruenberg, Cohomological topics in group theory,
  Lecture Notes in Mathematics \textbf{143}, Springer-Verlag,
  Berlin-New York, 1970.
%  
\bibitem{GrSe94} F.~Grunewald, D.~Segal, On affine crystallographic
  groups, J.\ Differential Geom.\ \textbf{40} (1994), 563--594.
%
%\bibitem{Ha11} H.~Hamrouni, Remarks on a conjecture of Chabauty,
%  Proc.\ Amer.\ Math.\ Soc.\ \textbf{139} (2011), 1983--1987.
% 
\bibitem{HoMo57} G.~Hochschild, G.D.~Mostow, Representations and
  representative functions of Lie groups, Ann.\ of Math.\ (2)
  \textbf{66} (1957), 495--542.
%
\bibitem{Ja79} N.~Jacobson, Lie algebras, Dover Publications, Inc.,
  New York, 1979.
% 
\bibitem{Ko97} Y~Komori, Semialgebraic description of Teichm\"uller
  space, Publ.\ Res.\ Inst.\ Math.\ Sci.\ \textbf{33} (1997),
  527--571.
%
\bibitem{Ma51} A.I.~Mal'tsev, On a class of homogeneous spaces, Amer.\
  Math.\ Soc.\ Translation \textbf{39} (1951), 1--33.
%
\bibitem{Mi51} E.~Michael, Topologies on spaces of subsets,
  Trans.\ Amer.\ Math.\ Soc.\ \textbf{71} (1951), 152--182.
%
\bibitem{Min1887} H.~Minkowski, Zur Theorie der positiven
  quadratischen Formen, J.\ reine angew.\ Math.\ \textbf{101} (1887),
  196--202.
%
\bibitem{Mo54} G.D.~Mostow, Factor spaces of solvable groups, Ann.\ of
  Math.\ \textbf{60} (1954), 1--27.
%
\bibitem{Mo57} G.D.~Mostow, On the fundamental group of a homogeneous
  space, Ann.\ of Math.\ (2) \textbf{66} (1957), 249--255.
%
\bibitem{Mo70} G.D.~Mostow, Representative functions on discrete
  groups and solvable arithmetic subgroups, Amer.\ J.\ Math.\
  \textbf{92} (1970), 1--32.
%
\bibitem{OnVi00} E.B.~Vinberg, V.V.~Gorbatsevich, O.V.~Shvartsman,
  Discrete subgroups of Lie groups, Lie groups and Lie algebras II,
  Encycl.\ Math.\ Sci.\ \textbf{21}, Springer-Verlag,
  Berlin-Heidelberg-New York, 2000.
%
\bibitem{PlRa94} V.~Platonov, A.~Rapinchuk, Algebraic groups and
  number theory, Pure and Applied Mathematics \textbf{139}, Academic
  Press, Inc., Boston, MA, 1994.
%
\bibitem{Ra72} M.S.~Raghunathan, Discrete subgroups of Lie groups,
  Ergebnisse der Mathematik und ihrer Grenzgebiete \textbf{68},
  Springer-Verlag, Berlin-Heidelberg-New York, 1972.
%
\bibitem{Sa57} M.~Sait\^o, Sur certains groupes de Lie r\'esolubles II,
  Sci.\ Papers Coll.\ Gen.\ Ed.\ Univ.\ Tokyo \textbf{7} (1957),
  157--168.
%
\bibitem{Se83} D.~Segal, Polycyclic groups, Cambridge Tracts in
  Mathematics \textbf{82}, Cambridge University Press, Cambridge,1983.
%
\bibitem{MaSi75} A.M.~Macbeath, D.~Singerman, Spaces of subgroups and
  Teichm\"uller space, Proc.\ London Math.\ Soc.\ \textbf{31} (1975),
  211--256.
% 
\bibitem{St94} A.N.~Starkov, Rigidity problem for lattices in solvable
  Lie groups, Proc.\ Indian Acad.\ Sci.\ (Math.\ Sci.)\ \textbf{104}
  (1994), 495--514.
%  
\bibitem{Wa63} H.-C.~Wang, On the deformations of lattice % typo in
  % the original paper!
  in a Lie group, Am.\ J.\ Math.\ \textbf{85} (1963), 189--212.
%
\bibitem{We60} A.~Weil, On discrete subgroups of Lie groups,
  Ann.\ Math.\ \textbf{72} (1960), 369--384.
%
\bibitem{Wh57} H.~Whitney, Elementary structure of real algebraic
  varieties, Ann.\ Math.\ \textbf{66} (1957), 545--556.
%
\bibitem{Wi00} B.~Wilking, Rigidity of group actions on solvable Lie
  groups, Math.\ Ann.\ \textbf{317} (2000), 195--237.
%
\bibitem{Wi95} D.~Witte, Superrigidity of lattices in solvable Lie
  groups, Invent.\ Math.\ \textbf{122} (1995), 147--193.
%  
\end{thebibliography}
\end{document}